\documentclass[reqno]{amsart}
\usepackage{amssymb, amsmath}
\usepackage{natbib}
\usepackage{lscape}
\usepackage{dsfont}
\usepackage[pdftex,plainpages=false,colorlinks,hyperindex,bookmarksopen,linkcolor=red,citecolor=blue,urlcolor=blue]{hyperref}

\DeclareMathAlphabet{\mathpzc}{OT1}{pzc}{m}{it}

\bibpunct{[}{]}{;}{n}{,}{,}

\newtheorem{te}{Theorem}[section]
\newtheorem{defin}[te]{Definition}
\newtheorem{os}[te]{Remark}
\newtheorem{prop}[te]{Proposition}
\newtheorem{lem}[te]{Lemma}

\numberwithin{equation}{section}

\allowdisplaybreaks

\newcommand {\puntomio} {\mathbin{\vcenter{\hbox{\scalebox{.45}{$\bullet$}}}}}

\def \l { \left( }
\def \r {\right) }
\def \ll { \left\lbrace }
\def \rr { \right\rbrace }

\begin{document}

	\title [] {Convolution-type derivatives, hitting-times of subordinators and time-changed $C_0$-semigroups}
	\author{Bruno Toaldo} 
	\address{Department of Statistical Sciences, Sapienza University of Rome}
	\email {bruno.toaldo@uniroma1.it}
	\keywords{Subordinator, inverse process, hitting-time, fractional calculus, time-changed semigroup, continuous time random walk, L\'evy measure}
	\date{\today}
	\subjclass[2000]{60G51, 60J35, 34K30}

		\begin{abstract}
This paper takes under consideration subordinators and their inverse processes (hitting-times). The governing equations of such processes are presented by means of convolution-type integro-differential operators similar to the fractional derivatives. Furthermore the concept of time-changed $C_0$-semigroup is discussed in case the time-change is performed by means of the hitting-time of a subordinator. Such time-change gives rise to bounded linear operators governed by integro-differential time-operators. Because these operators are non-local the presence of long-range dependence is investigated.
		\end{abstract}
	
	\maketitle

\tableofcontents

\section{Introduction}

The study of subordinators and their hitting-times has attracted the attention of many researchers since the 1940's. In particular a great effort has been dedicated to the study of the relationships between Bochner subordination and Cauchy problems (\citet{bochner, bochner2}). See \citet{fellerb, jacob1, librobern} and the references therein for more information on Bochner subordination. A subordinator $^f\sigma(t)$, $t\geq 0$, is a stochastic process with non-decreasing paths for which $\mathbb{E}e^{-\lambda \, ^f\sigma(t)}=e^{-tf(\lambda)}$ where $f$ is a Bernstein function (see \citet{bertoinb, bertoins} for more details on subordinators). Its inverse process is defined as
\begin{equation}
^fL (t) \, = \, \inf \ll s\geq 0 : \, ^f\sigma (s) > t \rr
\end{equation}
and is the hitting-time of $^f\sigma$. When the function $f$ is $f(\lambda) = \lambda^\alpha$, $\alpha \in (0,1)$, the related subordinator is called the $\alpha$-\emph{stable subordinator} and the inverse process $L^\alpha (t) = \inf \ll s>0 : \sigma^\alpha (s) > t \rr$ is called the \emph{inverse stable subordinator} (see \citet{meer12, stracca, samotaq} for more information on the stable subordinator and its inverse process). The relationships between such processes and partial differential equations have been object of intense study in the past three decades and have gained considerable popularity together with the study of fractional calculus (for fractional calculus the reader can consult \citet{kill}). As pointed out in \citet{orsptrf, orsann}, fractional PDEs are indeed related to time-changed processes while the relationships between time-fractional Cauchy problems and the inverse of the stable subordinator was explored for the first time by \citet{baem, meer09, saiche, zasla}. Equations of fractional order appear in a lot of physical phenomena (\citet{meer12}) and in particular for modeling anomalous diffusions (see for example \citet{benson, mirkospa}).

In the present paper we deal with the inverse processes $^fL(t)$, $t\geq 0$, of subordinators $^f\sigma (t)$, $t\geq 0$, with Laplace exponent the Bernstein function $f$ having the following representation
\begin{equation}
f(x) \, = \, a + bx + \int_0^\infty \l 1-e^{-sx} \r \, \bar{\nu}(ds)
\end{equation}
for a non-negative measure $\bar{\nu}$ on $(0, \infty)$ (\citet{artbern, librobern}). We consider the case in which the tail $s \to \nu(s)= a + \bar{\nu}(s, \infty)$ is absolutely continuous on $(0, \infty)$ and we define integro-differential operators similar to the fractional derivatives. In particular we show how the operator
\begin{equation}
^f\mathfrak{D}_t u(t) \, = \, b \frac{d}{dt} u(t) + \int_0^t \frac{\partial}{\partial t}  u(t-s) \, \nu(s) \, ds
\label{palla}
\end{equation}
allows us to write the governing equations of
\begin{align}
\mathcal{T}_tu \, = \, \int_0^\infty T_su \, l_t(ds), \qquad u \in \mathfrak{B},
\label{intro12}
\end{align}
where $l_t(B)=\Pr \ll \, ^fL(t) \in B \rr$ is the distribution of $^fL$ and $T_s$ is a $C_0$-semigroup on the Banach space $\l \mathfrak{B}, \left\| \cdot \right\|_\mathfrak{B} \r$. We call the operator $\mathcal{T}_t$ a time-changed $C_0$-semigroup.
In fact the main result of the present paper shows that $\mathcal{T}_tu$, $u \in \mathfrak{B}$, is a bounded strongly continuous linear operator on $\mathfrak{B}$ and solves the problem
\begin{equation}
\begin{cases}
^f\mathfrak{D}_t q(t) \, = \, A q(t), \qquad 0<t<\infty, \\
q(0) = u \in \textrm{Dom}\l A \r,
\end{cases}
\label{problemaintro}
\end{equation}
where $A$ is the infinitesimal generator of the $C_0$-semigroup $T_tu$, $u \in \mathfrak{B}$.

A central role in our analysis is played by the tail $\nu(s)$ of the L\'evy measure $\bar{\nu}$ since it emerges through all the results of the paper. It appears in the definitions of convolution-type derivatives of the form \eqref{palla} we will discuss in Section \ref{sezionederivate}. Furthermore we prove the following convergence in distribution
\begin{align}
\lim_{\gamma \to 0} \l bt + \sum_{j=1}^{N \l t \, \nu(\gamma) \r} Y_j \r \, \stackrel{\textrm{law}}{=} \, ^f\sigma(t), \qquad t\geq 0,
\end{align}
where $Y_j$ are i.i.d. random variables with distribution
\begin{align}
\Pr \ll Y_j \in dy \rr \, = \, \frac{1}{\nu(\gamma)} \l \bar{\nu}(dy) + a\delta_\infty \r \mathds{1}_{y> \gamma}, \qquad \gamma >0, \, \forall j = 1, \dots, n,
\end{align}
and $N(t)$, $t\geq 0$, is a homogeneous Poisson process independent from the r.v.'s $Y_j$. The symbol $\delta_\infty$ stands for the Dirac point mass at infinity.

\subsection*{List of symbols}
Here is a list of the most important notations adopted in the paper.
\begin{enumerate}
\item[$\bullet$] With $\mathcal{L} \left[ u(\puntomio) \right] (\lambda) \, = \, \widetilde{u}(\lambda)$ we denote the Laplace transform of the function $u$.
\item[$\bullet$] $\mathcal{F}\left[u(\puntomio) \right] (\xi) = \widehat{u}(\xi)$ indicates the Fourier transform of the function $u$.
\item[$\bullet$] With $^f\sigma (t)$, $t\geq 0$, we denote the subordinator with Laplace exponent $f$.
\item[$\bullet$] $\mu_t (B) = \Pr \ll \, ^f\sigma(t) \in B \rr$ indicates the convolution semigroup (transition probabilities) associated with the subordinator $^f\sigma(t)$, $t\geq 0$. When the measure $\mu_t$ has a density we adopt the abuse of notation $\mu_t(ds) = \mu_t(s)ds$ where $\mu_t(s)$ indicates the density of $\mu_t$.
\item[$\bullet$] $^fL(t)$, $t\geq 0$, indicates the inverse of the subordinator $^f\sigma(t)$, $t\geq 0$.
\item[$\bullet$] The symbol $l_t(B) = \Pr \ll \,  ^fL(t) \in B \rr$ indicates the distribution of $^fL(t)$, $t\geq 0$. With abuse of notation we denote by $l_t(s)$ the density of $l_t(ds)$.
\item[$\bullet$] With $A$ we denote the infinitesimal generator of the semigroup $T_tu$ for $u \in \mathcal{B}$ ($\mathcal{B}$ is a Banach space).
\end{enumerate}

\section{Convolution-type derivatives}
\label{sezionederivate}
In this section we define convolution-type operators similar to the fractional derivatives. The logic of our definitions starts from the observation of the fractional derivative of order $\alpha \in (0,1)$ (in the Riemann-Liouville sense) can be considered the first-order derivative of the Laplace convolution $u(t) * t^{-\alpha}/\Gamma(1-\alpha)$ (see \citet{kill})
\begin{equation}
\frac{d^\alpha}{d t^\alpha} u(t) \, = \, \frac{1}{\Gamma(1-\alpha)} \frac{d}{dt} \int_0^t \frac{u(s)}{(t-s)^\alpha} \, ds.
\label{lapeppa}
\end{equation}
For fractional calculus and applications the reader can also consult \citet{mainalibro} and for different form of fractional derivatives \citet{achar, kochu, lorenzo, meertri, meer12}.
Formula \eqref{lapeppa} can be formally viewed as $ \l \frac{d}{dt} \r^\alpha $ for $\alpha \in (0,1)$. Here we generalize this idea to a Bernstein function (\citet{artbern}). A Bernstein function is a function $f(x):(0, \infty) \to \mathbb{R}$ of class $C^\infty$, $f(x) \geq 0, \, \forall x >0$ for which
\begin{align}
(-1)^k f^{(k)} (x) \leq 0, \qquad \forall x > 0 \textrm{ and } k \in \mathbb{N}.
\end{align}
A function $f$ is a Bernstein function if, and only if, admits the representation
\begin{equation}
f(x) \, = \, a + bx + \int_0^\infty \l 1-e^{-s x} \r \, \bar{\nu} (ds), \qquad x >0,
\label{bernsteingenerica}
\end{equation}
where $a,b \geq 0$ and $\bar{\nu}(ds)$ is a non-negative measure on $(0, \infty)$ satisfying the integrability condition
\begin{equation}
\int_0^\infty \l z \wedge 1 \r \, \bar{\nu}(dz) < \infty.
\label{integrability}
\end{equation}
According to the literature we refer to the measure $\bar{\nu}$ and to the triplet $(a, b, \bar{\nu})$ as the \textit{L\'evy measure} and the \textit{L\'evy triplet} of the Bernstein function $f$. The representation \eqref{bernsteingenerica} is called the \textit{L\'evy-Khintchine} representation of $f$.

The Bernstein functions are closely related to the so-called completely monotone functions (see more on Bernstein function in \citet{jacob1, librobern}). The function $g(x) : (0, \infty) \to \mathbb{R}$ is completely monotone if has derivatives of all orders satisfying
\begin{align}
(-1)^k g^{(k)} (x) \geq 0, \qquad \forall x > 0 \textrm{ and } k \in \ll 0 \rr \cup \mathbb{N}.
\end{align}
By Bernstein Theorem (see \cite{artbern}) the function $g$ is completely monotone if and only if
\begin{equation}
g(x) \, = \, \int_0^\infty e^{-sx} m(ds), \qquad x > 0,
\end{equation}
when the above integral converges $\forall x > 0$ and where $m(ds)$ is a non-negative measure on $[0, \infty)$. Here and all throughout the paper the following symbology and definitions will be the same. We use $f(\puntomio)$ to denote the Bernstein function with representation \eqref{bernsteingenerica} and we consider the completely monotone function
\begin{equation}
g(x) \, = \, \frac{f(x)}{x}, \qquad x > 0,
\label{gi}
\end{equation}
with representation
\begin{align}
g(x) \, = \, b + \int_0^\infty e^{-sx} \, \nu(s)ds ,
\label{gimisura}
\end{align}
where $\nu (s)$ is the tail of the L\'evy measure appearing in \eqref{bernsteingenerica} 
\begin{align}
\nu(s)ds \, = \, \l a +  \, \bar{\nu}\l s, \infty  \r \r \, ds.
\label{finoaqui}
\end{align}
The representations \eqref{gi} and \eqref{gimisura} define a completely monotone function and are valid for every Bernstein function $f$ (see for example \citet{librobern} Corollary 3.7 \emph{(iv)}).
We observe that $\nu(s)$ is in general a right-continuous and non-increasing function for which
\begin{equation}
\int_0^1 \l a+ \bar{\nu}(s, \infty) \r ds \, = \, \int_0^1 \nu(s) \, ds \, < \infty.
\label{opiccolo}
\end{equation}
Furthermore we note that 
\begin{equation}
\bar{\nu} (s, \infty) < \infty, \, \textrm{for all } s > 0.
\label{giustific}
\end{equation}
In order to justify \eqref{giustific} we recall the inequality
\begin{align}
\l 1-e^{-1} \r \l t \wedge 1 \r \leq 1-e^{-t}, \qquad t \geq 0,
\end{align}
which can be extended as
\begin{align}
\l 1-e^{-\epsilon} \r \l t \wedge \epsilon \r \leq \l 1-e^{-t} \r, \quad \textrm{ for all } 0< \epsilon \leq 1, \, t \geq 0.
\label{altradiseguaglianza}
\end{align}
By taking into account \eqref{altradiseguaglianza} we can rewrite for all $0 < \epsilon \leq 1$ the integrability condition \eqref{integrability} as
\begin{equation}
\int_0^\infty \l t \wedge \epsilon \r \bar{\nu}(dt) < \infty, \qquad \textrm{ for all } 0 < \epsilon \leq 1,
\label{nuovaintegrability}
\end{equation}
since
\begin{equation}
\int_0^\infty \l t \wedge \epsilon \r \bar{\nu}(dt) \leq \frac{e^{\epsilon}}{e^{\epsilon}-1} \int_0^\infty \l 1-e^{-t} \r \, \bar{\nu}(dt) \, = \, \frac{e^{\epsilon}}{e^{\epsilon}-1} f(1) \, < \, \infty
\end{equation}
and this implies \eqref{giustific}.
When the L\'evy measure has finite mass, that is 
\begin{equation}
\bar{\nu}(0, \infty) < \infty,
\end{equation}
and if $b=0$, the corresponding Bernstein function $f$ is bounded.

\subsection{Convolution-type derivatives on the positive half-axis}
In this section we define a generalization, with respect to a Bernstein function $f$, of the classical Riemann-Liouville fractional derivative and we discuss some of its fundamental properties. Here is the first definition.
\begin{defin}
\label{definizionederivata}
Let $0<c\leq d<\infty$ and $u \in AC([c, d])$ that is the space of absolutely continuous function on $[c, d]$. Let $f$ be a Bernstein function with representation \eqref{bernsteingenerica} and let $\bar{\nu}$ be the corresponding L\'evy measure with tail $\nu(s)=a + \bar{\nu}(s, \infty)$. Assume that $s \to \nu(s)$ is absolutely continuous on $(0, \infty)$. We define the generalized Riemann-Liouville derivative according to the Bernstein function $f$ as
\begin{align}
^f\mathcal{D}_t^{(c, d)} u(t) \, : = \, \frac{d}{dt} \left[ b  u(t) +   \int_0^{t-c} u(t-s) \,\nu(s)ds \right], \qquad t \in [c, d].
\label{formuladef}
\end{align}
\end{defin}
The representation \eqref{formuladef} can be extenended to define the derivative on the half-axis $\mathbb{R}^+$ as it is done for the classical Riemann-Liouville fractional derivative (see \citet{kill} page 79). Hence we write
\begin{align}
^f\mathcal{D}_t^{(0,\infty )} u(t) \, : = \, \frac{d}{dt} \left[ b  u(t) +   \int_0^{t} u(t-s) \,\nu(s)ds \right].
\label{nonsenepuopiu}
\end{align}

\begin{lem}
\label{laplacederivata}
Let $^f\mathcal{D}_t^{(c, \infty)}u(t)$, $t \geq c \geq 0$, be as in Definition \ref{definizionederivata} and let $|u(t)| \leq Me^{\lambda_0 t}$ for some $\lambda_0, M>0$. We have the following result
\begin{align}
\mathcal{L} \left[ \, ^f\mathcal{D}_t^{(c, +\infty)} u(t) \right] (\lambda) \, = \, f(\lambda) \, \widetilde{u}(\lambda) - b e^{-\lambda c} u(c), \qquad \Re\lambda > \lambda_0.
\end{align}
\end{lem}
\begin{proof}
The Laplace transform can be evaluated explicitely as follows
\begin{align}
\mathcal{L} \left[ ^f\mathcal{D}_t^{(c, +\infty)} u(t) \right] (\lambda) \, = \, & b \lambda \widetilde{u} (\lambda) -be^{-\lambda c} u(c) +  \mathcal{L} \left[ \frac{d}{dt} \int_0^{t-c} u(t-s) \, \nu(s)ds \right] (\lambda) \notag \\
= \, & b\lambda \widetilde{u} (\lambda) -be^{-\lambda c} u(c) +   \lambda \mathcal{L} \left[ \int_0^{t-c} u(t-s) \, \nu(s)ds \right] (\lambda) \notag \\
= \, & b\lambda \widetilde{u} (\lambda) -be^{-\lambda c} u(c) +  \lambda \int_0^{\infty} \int_{s+c}^\infty e^{-\lambda t} u(t-s) \, \nu(s) dt \, ds \notag \\
= \, & \lambda g(\lambda) \widetilde{u} (\lambda) - be^{-\lambda c} u(c) \notag \\
= \, & f(\lambda) \,  \widetilde{u} (\lambda) -be^{-\lambda c} u(c).
\end{align}
In the last steps we used \eqref{gi} and \eqref{gimisura}.
\end{proof}
In view of the previous Lemma we note that our definition is consistent and generalizes the Riemann-Liouville fractional derivatives of order $\alpha \in (0,1)$ in a reasonable way.
\begin{os}
\label{coincidriemann}
Let the function $f$ of Definition \ref{definizionederivata} be $f(x) \, = \, x^\alpha$, $x>0$, $\alpha \in (0,1)$, for which \eqref{bernsteingenerica} becomes
\begin{equation}
x^\alpha \, = \, \int_0^\infty \l 1-e^{-sx} \r \frac{\alpha s^{-\alpha -1}}{\Gamma (1-\alpha)} ds,
\label{bernsteinxallaalfa}
\end{equation}
that is to say $a=0$ and $b=0$ and
\begin{align}
\bar{\nu}(ds) \, = \, \frac{\alpha s^{-\alpha -1}}{\Gamma (1-\alpha)} ds
\end{align}
and therefore
\begin{align}
\nu(s)ds \, = \, ds \int_s^\infty \frac{\alpha z^{-\alpha -1}}{\Gamma (1-\alpha)} dz \, = \, \frac{s^{-\alpha}ds}{\Gamma (1-\alpha)}.
\end{align}
By performing these substitutions in Definition \ref{definizionederivata} it is easy to show that
\begin{align}
^f\mathcal{D}_t^{(0, +\infty)} u(t) \, = \, \frac{^Rd^\alpha}{d t^\alpha} u(t)
\end{align}
where
\begin{align}
\frac{^Rd^\alpha}{dt^\alpha} u(t) \, = \, \frac{1}{\Gamma (1-\alpha)} \frac{d}{dt} \int_0^t \frac{u(s)}{(t-s)^\alpha} ds
\label{definizioneriemann}
\end{align}
is the Riemann-Liouville fractional derivative.
\end{os}

By following the logic inspiring the fractional Dzerbayshan-Caputo derivative (see \cite{kill}) defined, for an absolutely continuous function $u(t)$, $t>0$, as
\begin{align}
\frac{^Cd^\alpha}{dt^\alpha} u(t) \, = \, \frac{1}{\Gamma (1-\alpha)} \int_0^t \frac{u^\prime(s)}{(t-s)^\alpha} ds,
\label{definizionecaputo}
\end{align}
we can give the following alternative definition of generalized derivative with respect to a Bernstein function.
\begin{defin}
\label{definizaltern}
Let $0<c\leq d < \infty$ and $u \in AC([c, d])$. Let $f$ and $\nu$ be as in Definition \ref{definizionederivata} and $u (t) \in AC\l [c, d] \r$. We define the generalized Dzerbayshan-Caputo derivative according to the Bernstein function $f$ as
\begin{align}
^f\mathfrak{D}_t^{(c, d)} u(t) \, : = \, b \frac{d}{d t} u(t)+   \int_0^{t-c} \frac{\partial}{\partial t} u(t-s) \, \nu (s) ds, \qquad t \in [c, d].
\label{formuladefalterna}
\end{align}
\end{defin}
As already done for the classical Dzerbayshan-Caputo derivative we can extend \eqref{formuladefalterna} to the half-axis $\mathbb{R}^+$ (see for example \cite{kill} page 97) by
\begin{equation}
^f\mathfrak{D}_t^{(0, \infty)} u(t) \, : = \, b \frac{d}{dt} u(t)+   \int_0^{t} \frac{\partial}{\partial t} u(t-s) \, \nu (s) ds.
\label{alternativaestesa}
\end{equation}
Throughout the paper we will write for the sake of simplicity $^f\mathfrak{D}_t$ instead of $^f\mathfrak{D}_t^{(0, \infty)}$.
\begin{lem}
\label{lemmaplacealtern}
Let $^f\mathfrak{D}_t$ be as in \eqref{alternativaestesa} and let $|u(t)| \leq M e^{\lambda_0 t}$, for some $\lambda_0$, $M>0$. We obtain
\begin{align}
\mathcal{L} \left[ ^f\mathfrak{D}_t u(t) \right] (\lambda) \, = \, f(\lambda) \widetilde{u}(\lambda) - \frac{f(\lambda)}{\lambda} u(0), \qquad \Re \lambda > \lambda_0.
\label{3222}
\end{align}
\end{lem}
\begin{proof}
By evaluating explicitely the Laplace transform we obtain
\begin{align}
\mathcal{L} \left[\, ^f\mathfrak{D}_t u(t) \right] (\lambda) \, = \, & b \lambda \widetilde{u}(\lambda) - bu(0) + \int_0^\infty e^{-\lambda t} \int_0^t  \frac{d}{dt} u(t-s) \nu(s)ds \, dt \notag \\
= \, &  b \lambda \widetilde{u}(\lambda) - bu(0) + \int_0^\infty \int_s^\infty e^{-\lambda t} \frac{d}{dt} u(t-s) \, \nu(s) \, dt \, ds \notag \\
= \, &  b \lambda \widetilde{u}(\lambda) - bu(0) + \int_0^\infty e^{-\lambda s} \nu(s)ds \l \lambda \widetilde{u}(\lambda) - u(0) \r \notag \\
= \, & \lambda g(\lambda) \widetilde{u} (\lambda)- g(\lambda) u(0) \notag \\
= \, & f(\lambda) \widetilde{u} (\lambda) - \frac{f(\lambda)}{\lambda} u(0)
\end{align} 
where we used the relationships \eqref{gi} and \eqref{gimisura}.
\end{proof}
\begin{os}
By performing the same substitutions of Remark \ref{coincidriemann} it is easy to show that
\begin{align}
^f\mathfrak{D}_t u(t) \, = \, \frac{^Cd^\alpha}{dt^\alpha} u(t)
\end{align}
where $\frac{^Cd^\alpha}{dt^\alpha}$ is the Dzerbayshan-Caputo derivative defined in \eqref{definizionecaputo}.
\end{os}
It is well known that the Riemann-Liouville fractional derivative of a function $u \in AC([c, d)]$ exist almost everywhere in $[c, d]$ and can be written as (see \citet{kill} page 73)
\begin{align}
\frac{^Rd^\alpha}{dt^\alpha} u(t) \, = \, \frac{^Cd^\alpha}{dt^\alpha} u(t) + \frac{(t-c)^{-\alpha}}{\Gamma (1-\alpha)}  u(c).
\end{align}
Here is a more general result.
\begin{prop}
Let $^f\mathcal{D}_t^{(c, d)}$ and $^f\mathfrak{D}_t^{(c, d)}$ be respectively as in Definitions \ref{definizionederivata} and \ref{definizaltern}. We have that $^f\mathcal{D}_t^{(c, d)}u(t)$ exists almost everywhere in $[c, d]$ and can be written as
\begin{align}
^f\mathcal{D}_t^{(c, d)} u(t) \, = \, ^f\mathfrak{D}_t^{(c, d)} u(t) + \nu(t-c) u(c).
\label{tipocaputoriemann}
\end{align}
\end{prop}
\begin{proof} 
Let 
\begin{align}
V(s) = \int_0^s \nu(w) dw
\end{align}
which is convergent in view of \eqref{opiccolo}.
Since $u \in AC([c, d])$ we have for $c<s<d$
\begin{align}
u(s) = \int_c^s u^\prime (z) dz + u(c)
\end{align}
and therefore we can rewrite $^f\mathcal{D}^{(c, d)}_tu(t)$ as
\begin{align}
& ^f\mathcal{D}_t^{(c, d)}u(t) \notag \\
= \, & b\frac{d}{dt}u(t) + \frac{d}{dt} \int_0^{t-c}u(t-s) \nu(s)ds \notag \\
= \, & b\frac{d}{dt}u(t) + \frac{d}{dt}\int_c^t u(s) \nu(t-s)ds \notag \\
 = \, &b\frac{d}{dt}u(t) + \frac{d}{dt}\int_c^t \l \int_c^s u^\prime (z) dz + u(c) \r \nu(t-s)ds \notag \\
 = \, & b\frac{d}{dt}u(t) +\frac{d}{dt} \left[ -V(t-s) \l \int_c^s u^\prime (z)dz + u(c) \r \right]_{s=c}^{s=t} +\frac{d}{dt} \int_c^t V(t-s) u^\prime(s) ds \notag \\
 = \, & b\frac{d}{dt}u(t) + \frac{d}{dt} \left[-V(0) \int_c^t u^\prime (z)dz - u(c) + V(t-c)u(c) \right]\notag \\
 & +V(0) u^\prime(t)+\int_c^t u^\prime (s) \nu(t-s)ds \notag \\
 = \, & b\frac{d}{dt}u(t)+ \nu(t-c)u(c) + \int_c^t \frac{d}{ds}u(s) \; \nu(t-s)ds \notag \\
 = \, & ^f\mathfrak{D}^{(c,d)}_t \, u(t) + \nu(t-c)u(c).
\end{align}
In the fourth step we performed an integration by parts.
\end{proof}

\subsection{Convolution-type derivatives on the whole real axis}
In this section we develop a generalized space-derivative with respect to a Bernstein function $f$ with domain on the whole real axis $\mathbb{R}$, by following the logic inspiring the Weyl derivatives.
\begin{defin}
\label{definderivataspace}
Let $f$ and $\nu(s)$ be as in Definition \ref{definizionederivata}. We define the generalized Weyl derivative, according to the Bernstein function $f$, on the whole real axis as
\begin{align}
^f\mathpzc{D}^+_x u(x) \, := \,  \left[ b \frac{d}{dx} u(x) +  \int_{0}^\infty \, \frac{\partial}{\partial x} u(x - s) \,  \nu(s)ds \right], \qquad x \in \mathbb{R},
\label{miadestra}
\end{align}
and
\begin{equation}
^f\mathpzc{D}^-_x u(x) \, := \, - \left[ b \frac{d}{dx}  u(x)  + \int_{0}^\infty \, \frac{\partial}{\partial x} u(x+s) \,  \nu(s)ds \right], \qquad x \in \mathbb{R}.
\label{miasinistra}
\end{equation}
\end{defin}
Some remarks on the domain of definition of \eqref{miadestra} and \eqref{miasinistra} are stated in Section \ref{sezionefilippo}.
\begin{lem}
\label{fourierderivataspace}
Let $^f\mathpzc{D}^{\pm}_x$ be as in Definition \ref{definderivataspace}. We have that
\begin{equation}
\mathcal{F} \left[ \, ^f\mathpzc{D}_x^+ u(x) \right] (\xi) \, = \, f(-i\xi) \widehat{u}(\xi)
\end{equation}
and
\begin{align}
\mathcal{F} \left[ \, ^f\mathpzc{D}_x^- u(x) \right] (\xi) \, = \, f(i\xi) \widehat{u}(\xi).
\label{repetition}
\end{align}
\end{lem}
\begin{proof}
By evaluating the first Fourier transform explicitely, we obtain
\begin{align}
\mathcal{F} \left[ ^f\mathpzc{D}^{+}_x u(x) \right] (\xi) \, = \, & -bi\xi \widehat{u}(\xi)  - i\xi \, \mathcal{F} \left[  \int_0^\infty u(x - s) \nu(s)ds \right] (\xi) \notag \\
= \, &  -bi\xi \widehat{u}(\xi) -i\xi \int_0^\infty \int_{\mathbb{R}} e^{i\xi z+i\xi s} u(z) \, dz \, \nu(s)ds \notag \\
= \, & -bi\xi \widehat{u}(\xi) - i \xi \int_0^\infty ds  \, e^{i\xi s} \, \l a+ \int_s^\infty \bar{\nu}(dz) \r \widehat{u}(\xi)
\end{align}
and by integrating by parts we get that
\begin{align}
\mathcal{F} \left[ ^f\mathpzc{D}^{+}_x u(x) \right] (\xi) \, = \, & a\widehat{u}(\xi)-bi\xi \widehat{u}(\xi) + \int_0^\infty \l 1-e^{i\xi s} \r \bar{\nu} (ds) \, \widehat{u}(\xi) \notag \\
= \, & f(-i\xi) \, \widehat{u}(\xi).
\end{align}
By repeating the same calculation one can easily prove \eqref{repetition}.
\end{proof}
\begin{os}
Definitions \eqref{miadestra} and \eqref{miasinistra} are consistent with the Weyl definition of fractional derivatives on the whole real axis which are, for $\alpha \in (0,1)$ and $x \in \mathbb{R}$, (see \cite{kill})
\begin{equation}
\frac{^+d^\alpha }{d x^\alpha}u(x) \, = \, \frac{1}{\Gamma (1-\alpha)} \frac{d}{dx} \int_{-\infty}^x \frac{u(s)}{(x-s)^{\alpha}} ds, \quad \textrm{right derivative},
\label{weyldestra}
\end{equation}
and
\begin{align}
\frac{^-d^\alpha }{d x^\alpha}u(x) \, = \, -\frac{1}{\Gamma (1-\alpha)} \frac{d}{dx} \int_x^\infty \frac{u(s)}{(s-x)^{\alpha}} ds,
\label{weylsinistra} \quad \textrm{left derivative.}
\end{align}
We have
\begin{align}
^f\mathpzc{D}_x^{\pm} u(x) \, = \,  \frac{^{\pm}d^\alpha }{d x^\alpha}u(x), \qquad x \in \mathbb{R}.
\label{243}
\end{align}
We resort to the fact that (see \cite{kill} page 90)
\begin{equation}
\mathcal{F} \left[ \frac{^{\pm}\partial^\alpha}{\partial x^\alpha} u(x) \right] (\xi) \, = \, (\mp i\xi)^\alpha \widehat{u}(\xi)
\label{fourierdiweyl}
\end{equation}
and thus by combining \eqref{fourierdiweyl} with Lemma \ref{fourierderivataspace} the proof of \eqref{243} is complete. The reader can also check the result by performing the substitution $b=0$ and
\begin{equation}
\nu(s)ds \, = \, \frac{s^{-\alpha}}{\Gamma (1-\alpha)} ds
\end{equation}
in \eqref{miadestra} and \eqref{miasinistra} which yields \eqref{weyldestra} and \eqref{weylsinistra} with a change of variable.
\end{os}

\section{Subordinators, hitting-times and continuous time random walks}
A subordinator $^f\sigma (t)$, $t\geq 0$, is a stochastic process in continuous time with non-decreasing paths (see more on subordinators in \citet{bertoinb, bertoins}) and values in $[0, \infty]$ where $\infty$ is an absorbing state (cemetery). The process $\sigma^f(t)$, $t \geq 0$, is a subordinator if it has independent and homogeneous increments on $[0, \zeta)$ where
\begin{equation}
\zeta \, = \, \inf \ll t \geq 0 : \sigma^f(t) = \infty \rr.
\label{lifetime}
\end{equation}
If $\zeta \stackrel{\textrm{a.s.}}{=} \infty$ the process $\sigma^f$ is said to be a strict subordinator since it has stationary and independent increments in the ordinary sense. In this case the Laplace exponent of $^f\sigma$ is
\begin{align}
f(\lambda) \, = \, b\lambda + \int_0^\infty \l 1-e^{-\lambda s} \r \bar{\nu}(ds)
\end{align}
that is $a=0$.

The transition probabilities of subordinators $\mu_t(B) = \Pr \ll \, ^f\sigma (t) \in B \rr$, $B \subset [0, \infty)$ Borel, $t>0$, are convolution semigroups of sub-probability measure with the following property concerning the Laplace transform
\begin{equation}
\mathcal{L} \left[ \mu_t \right] (\lambda) \, = \, e^{-tf(\lambda)}
\label{semigruppoconvoluz}
\end{equation}
where $f$ is a Bernstein function having representation \eqref{bernsteingenerica}.
A family $\mu_t$, $t>0$, of sub-probability measures on $\mathbb{R}^n$ is called a convolution semigroup on $\mathbb{R}^n$ if it satisfies the conditions
\begin{enumerate}
\item[$\bullet$] $\mu_t \l \mathbb{R}^n \r \leq 1$, $\forall t \geq 0$;
\item[$\bullet$] $\mu_s * \mu_t \, = \, \mu_{t+s}$, $\forall s, t \geq 0$, and $\mu_0 = \delta_0$;
\item[$\bullet$] $\mu_t \to \delta_0$, vaguely as $t \to 0$,
\end{enumerate}
where we denoted by $\delta_0$ the Dirac point mass at zero. The fact that the tail function $s \to \nu(s)$ of the L\'evy measure $\bar{\nu}$ is absolutely continuous on $(0, \infty)$ and that $\bar{\nu}(0, \infty) = \infty$ is a sufficient condition for saying the transition probabilities of the corresponding subordinator are absolutely continuous (see \citet{satolevy}, Theorem 27.7). 

It has been shown that any subordinator has a Laplace exponent as in \eqref{semigruppoconvoluz} and that any Bernstein function with representation \eqref{bernsteingenerica} is the Laplace exponent of a subordinator (see for example \cite{bertoins}). A subordinator is a step process if its associated Bernstein function $f$ is bounded. Looking at the representation \eqref{bernsteingenerica} we see that a Bernstein function is bounded if $\bar{\nu}(0, \infty) < \infty$ and $b=0$. If these conditions are not fulfilled (and thus $b>0$ and $\bar{\nu}(0, \infty) = \infty$) the subordinator is a strictly increasing process. 

The inverse process of a subordinator is defined as
\begin{equation}
^fL (t) \, = \, \inf \ll s > 0 : \, ^f\sigma (s) > t \rr, \qquad s,t >0,
\label{definizinverso}
\end{equation}
and thus $^fL$ is the hitting-time of $^f\sigma$ since $^f\sigma$ has non-decreasing paths (see \citet{bertoinb, bertoins}). With this in hand we note that $^fL$ is again a non-decreasing process but in general it has non-stationary and non-independent increments. In what follows we develop some properties of the distribution of $^fL(t)$, $t\geq 0$, denoted by $l_t(B) = \Pr \ll \, ^fL(t) \in B \rr$.
\begin{lem}
Let $^f\sigma (t)$, $t\geq 0$, and $^fL(t)$, $t\geq 0$, be respectively a subordinator and its inverse. Let $f$ be the Laplace exponent of $\, ^f\sigma$ represented as in \eqref{bernsteingenerica} for $a, b \geq 0$. Let $\nu(s)$ be the tail of the L\'evy measure $\bar{\nu}$ and $l_t(B)$ the distribution of $^fL$. Suppose that $s \to \nu(s)$ is absolutely continuous and that $\bar{\nu}(0, \infty)=\infty$. We have that
\begin{equation}
\mathcal{L} \left[ l_{\puntomio}(s, \infty) \right] ( \lambda) = \frac{1}{\lambda} e^{-sf(\lambda)}.
\end{equation}
\end{lem}
\begin{proof}
We resort to the fact that $^f\sigma$ has non-decreasing paths and thus, in view of the construction \eqref{definizinverso} of $^fL$ we have
\begin{align}
\Pr \ll ^fL(t) > s \rr \, = \, \Pr \ll ^f\sigma (s) < t \rr.
\label{probsubinv}
\end{align}
In view of \eqref{probsubinv} we observe that
\begin{align}
\int_0^\infty e^{-\lambda t} l_t(s, \infty] \, dt \, = \, \int_0^\infty e^{-\lambda t} \mu_s[0, t) dt
\end{align}
and thus
\begin{equation}
\int_0^\infty e^{-\lambda t} l_t[s, \infty) dt \, = \, \int_0^\infty e^{-\lambda t} \int_0^t \mu_s(dz) \, dt \, = \, \frac{1}{\lambda} e^{-sf(\lambda)}.
\label{cheroba}
\end{equation}
\end{proof}
\begin{prop}
\label{lemmasullinverso}
Let $^f\sigma (t)$, $t\geq 0$, be the subordinator with Laplace exponent $f$ represented by \eqref{bernsteingenerica} for $a \geq 0$, $b \geq 0$. Let $\nu$ be the tail of the L\'evy measure $\bar{\nu}$. Assume that $\bar{\nu}(0, \infty)=\infty$ and that $ s \to \nu(s) = a+ \bar{\nu}(s, \infty)$ is absolutely continuous on $(0, \infty)$. Let $^fL(t)$, $t\geq 0$, be the inverse of $^f\sigma$, in the sense of \eqref{definizinverso}, with distribution $l_t(B) = \Pr \ll \, ^fL(t) \in B \rr$. We have the following results.
\begin{enumerate}
\item The distribution $l_t$ have a density such that $l_t(ds) = l_t(s)ds$ and $l_t(s) = b\mu_s(t) + (\nu(t) *  \mu_s(t))$ where with abuse of notation we denoted with $l_t(s)$ and $\mu_s(t)$ respectively the density of $l_t(ds)$ and $\mu_s(dt)$ and the symbol $*$ stands for the Laplace convolution $\int_0^t \mu_s(t-z) \nu(z) dz$. Furthermore $\mathcal{L} \left[ l_{\puntomio}(s) \right] (\lambda) = \frac{f(\lambda)}{\lambda} e^{-sf(\lambda)}$.
\item $\lim_{h \to 0} l_{t+h} = l_t$ $\forall t\geq 0$ and $\lim_{t \to 0} l_t[0, \infty) = \delta_0[0, \infty)$.
\item $l_t(0) = \nu(t)$, $\forall t>0$.
\item $l_t[0, \infty) = 1$, $\forall a,b \geq 0$.
\end{enumerate}
\end{prop}
\begin{proof}
\begin{enumerate}
\item Since we assume $\bar{\nu}(0, \infty)=\infty$ and $s \to \nu(s)$ is absolutely continuous on $(0, \infty)$, we have that from Theorem 27.7 in \cite{satolevy} the transition probabilities $\mu_t(dx)$ are absolutely continuous and therefore have a density $\mu_t(x)$. Thus we write
\begin{align}
\mathcal{L}\left[ b\mu_s(\puntomio) + \l \mu_s(\puntomio) * \nu(\puntomio) \r \right]( \lambda) \, = \, & be^{-sf(\lambda)}+ \int_0^\infty e^{-\lambda t} \int_0^t \mu_s(t-z) \nu(z) dz \, dt \notag \\
 = \, & be^{-sf(\lambda)}+ \int_0^\infty dz  \int_z^\infty dt \, e^{-\lambda t} \, \mu_s(t-z)  \nu(z) \notag \\
 = \, & \frac{f(\lambda)}{\lambda} e^{-sf(\lambda)},
 \label{flour}
\end{align}
where we used \eqref{gimisura}. From \eqref{flour} we get
\begin{align}
  \int_s^\infty \mathcal{L} \left[ b \mu_w(\puntomio) + \l \mu_w(\puntomio) * \nu(\puntomio) \r  \right]  (\lambda) \, dw \, = \, \int_s^\infty \frac{f(\lambda)}{\lambda} e^{-wf(\lambda)} dw \, = \, \frac{1}{\lambda} e^{-sf(\lambda)}.
 \label{dovrebbe}
\end{align}
Since \eqref{dovrebbe} coincides with \eqref{cheroba} we can write
\begin{align}
\int_s^\infty \l b \mu_w(t)+ ( \mu_w(t) * \nu(t)) \r dw \, = \, l_t(s, \infty)
\end{align}
which completes the proof. 
\item We have
\begin{align}
\lim_{h \to 0} l_{t+h}[s, \infty) \, = \, & \lim_{h \to 0} \int_s^\infty \l b \mu_s(t+h) + \int_0^{t+h} \mu_s(t+h-z) \nu(z) dz \r ds \notag \\
 = \, &  l_t[s, \infty)
\end{align}
since $\mu_s(t)$ is a density. Furthermore
\begin{align}
\lim_{t \downarrow 0}   l_t[0, \infty) \, = \, \lim_{t \downarrow 0} \int_0^\infty \l b \mu_s(t) + \int_0^t \mu_s(t-z) \nu(z) dz \r ds \, = \, \delta_0[0, \infty).
\end{align}
\item This is obvious since for $t>0$, $l_t(0) \, = \, b\mu_0(t) + \nu(t) * \mu_0 (t) = \nu(t) $.
\item The proof of this can be carried out by observing that
\begin{align}
\int_0^\infty e^{-\lambda t } l_t[0, \infty) \, dt \, = \, \int_0^\infty  e^{-\lambda t} l_t[s, \infty) \, dt  \bigg|_{s = 0} \, = \, \frac{1}{\lambda} e^{-sf(\lambda)} \bigg|_{s=0} \, = \, \frac{1}{\lambda}.
\end{align}
\end{enumerate}
\end{proof}

Subordinators are related to Continuous Time Random Walks (CTRWs). The CTRWs (introduced in \citet{montrol}) are processes in continuous time in which the number of jumps performed in a certain amount of time $t$ is a random variable, as well as the jump's length. For example, the stable subordinator can be viewed (in distribution) as the limit of a CTRW performing a Poissonian number of power-law jumps (see for example \citet{meer12}). In \citet{meer04}, among other things, the authors pointed out that the limit process of a CTRW with infinite-mean waiting times converge to a L\'evy motion time-changed by means of the hitting-time $L^\alpha (t)$, $t\geq 0$, of the stable subordinator $\sigma^\alpha(t)$, $t\geq 0$. Since subordinators are also L\'evy processes they can be decomposed according to the L\'evy-It\^{o} decomposition (\citet{itodec}). By following the logic of the L\'evy-It\^{o} decomposition we derive a CTRW converging (in distribution) to a subordinator with laplace exponent $f$ and having a hitting-time converging to its inverse. Our CTRW is therefore the sum of a pure drift and a compound Poisson. The distribution of the jumps' length need some attention. In particular we define i.i.d. random variables $Y_j$ representing the random length of the jump, with law
\begin{align}
p_{Y_j}(dy) \, = \,  \frac{1}{\nu(\gamma)} \, \l \bar{\nu}(dy) + a \, \delta_\infty  \r \,  \mathds{1}_{y>\gamma}  , \qquad \gamma > 0, \, \forall j = 1, \cdots, n,
\label{leggedelley}
\end{align}
where $\delta_\infty$ indicates the Dirac point mass at $\infty$ and $a \geq 0$.
In \eqref{leggedelley} $\bar{\nu}$ and $\nu$ are respectively the L\'evy measure and its tail as defined in equations from \eqref{bernsteingenerica} to \eqref{finoaqui} and upon which the definitions of convolution-type derivatives of previous section are based. The parameter $a \geq 0$ is that in \eqref{bernsteingenerica} and it is known in literature as the killing rate of the subordinator. The distribution \eqref{leggedelley} can be taken as follows. The probability of a jump of length $y>\gamma > 0$ is given by the normalized L\'evy measure when $a = 0$. When $a>0$ the probability of a jump of infinite length increases since $\bar{\nu}(y) \stackrel{y \to \infty}{\longrightarrow}0$ and thus $\Pr \ll Y \in dy \rr/dy \stackrel{y \to \infty}{\longrightarrow} a / \nu(\gamma)$. When constructing a CTRW with Poisson waiting times and jump length's distribution \eqref{leggedelley} by choosing $a>0$ we obtain a limit process (for $\gamma \to 0$) assuming value $+\infty$ from a certain time $\zeta<\infty$ on. Usually $\zeta$ is called the lifetime of the process (see \cite{bertoins}). The case $a >0$ in \eqref{leggedelley} therefore gives rise to the so-called killed subordinators. A killed subordinator $^f\widehat{\sigma}_t$, is defined as
\begin{align}
^f\widehat{\sigma}_t \, = \, \begin{cases} ^f\sigma_t, \qquad & t < \zeta, \\ +\infty, & t \geq \zeta, \end{cases}
\end{align}
where $\zeta$ is the lifetime defined in \eqref{lifetime}.
Obviously $a=0$ implies $\zeta = \infty$.
For simplicity we will use the notation $^f\sigma_t$ both for killed and non-killed subordinators when no confusion arises.
We are ready to prove the following convergences in distribution inspired by the L\'evy-Ito decomposition and usefull in order to understand the role of the L\'evy measure $\bar{\nu}$ and its tail $\nu(s)$.

\begin{prop}
\label{proposizrandomwalk}
Let $N(t)$, $t\geq 0$, be a homogeneous Poisson process with parameter $\theta = 1$ independent from the i.i.d. random variables $Y_j$ with distribution \eqref{leggedelley}. Let $f$ be the Bernstein function with representation \eqref{bernsteingenerica} Laplace exponent of the subordinator $^f\sigma (t)$, $t\geq 0$, and let $^fL(t)$, $t\geq 0$ be the inverse of $^f\sigma$ as in \eqref{definizinverso}. Let $\nu(s)$ be the tail of the L\'evy measure $\bar{\nu}$. The following convergences in distribution are true.
\begin{enumerate}
\item 
\begin{align}
 \l bt + \sum_{j=0}^{N \l t \, \nu(\gamma) \r} Y_j \r \, \stackrel{\textrm{law}}{\longrightarrow} \, ^f\sigma(t) \textrm{ as }\gamma \to 0,
\label{convsubord}
\end{align}
\item 
\begin{equation}
\inf \ll s>0 : bs + \sum_{j=0}^{N \l s \nu(\gamma) \r} Y_j > t \rr \stackrel{\textrm{law}}{\longrightarrow}  \, ^fL(t) \textrm{ as }\gamma \to 0.
\end{equation}
\end{enumerate}
\end{prop}
\begin{proof}
In order to prove (1) we consider the following Laplace transform
\begin{align}
& \mathbb{E}\exp\ll {-\lambda bt-\lambda \sum_{j=0}^{N \l t \, \nu(\gamma) \r}Y_j} \rr \notag \\
  = \, & e^{-\lambda bt} \mathbb{E} \left[  \mathbb{E} \l e^{-\lambda Y} \r^{N \l t\nu(\gamma) \r}   \right] \notag \\
= \, & \exp \ll {-\lambda bt} e^{- t \nu(\gamma) \l 1-\mathbb{E}e^{-\lambda Y} \r} \rr \notag \\
= \, & \exp\ll {-t \l b\lambda + \nu(\gamma) \int_\gamma^\infty \l 1-e^{-\lambda y} \r p_{Y_j}(dy) \r} \rr,
\label{68}
\end{align}
where $p_{Y_j}(dy)$ is the one in \eqref{leggedelley}. In the previous steps we used the independence of the random variables $Y_j$ and the fact that
\begin{equation}
\mathbb{E}e^{-\lambda N \l t \nu(\gamma) \r } \, = \, e^{- t \nu(\gamma) \l 1-e^{- \lambda} \r}.
\end{equation}
By performing the limit for $\gamma \to 0$ in \eqref{68} we obtain
\begin{align}
& \lim_{\gamma \to 0}  \mathbb{E}\exp \ll {-\lambda bt-\lambda \sum_{j=0}^{N \l t \, \nu(\gamma) \r}Y_j} \rr \notag \\
 = \, & \exp \ll {-t \l a+ b\lambda +  \int_0^\infty \l 1-e^{-\lambda y} \r   \bar{\nu}(dy)  \r} \rr \notag \\
 = \, & e^{-tf(\lambda)},
\end{align}
and this proves (1).

Now we prove (2). Let $Z(t) = \inf \ll s>0 : bs + \sum_{j=0}^{N \l s \nu(\gamma) \r} Y_j > t \rr$. By definition we have that
\begin{align}
\Pr \ll Z(t)  > s \rr \, = \, \Pr \ll bs+\sum_{j=0}^{N(s \, \nu(\gamma))} Y_j < t \rr 
\end{align}
and thus
\begin{align}
\mathcal{L} \left[ \Pr \ll Z(\puntomio)  > s \rr \right] (\lambda) \, = \, & \mathcal{L} \left[ \Pr \ll bs+\sum_{j=0}^{N(s \, \nu(\gamma))} Y_j < \puntomio \rr  \right] (\lambda).
\end{align}

By taking profit of calculation \eqref{68} we obtain
\begin{align}
& \mathcal{L} \left[ \Pr \ll Z(\puntomio) > s \rr \right] (\lambda) \, = \, \frac{1}{\lambda} \exp\ll {-s \l a+ b\lambda +  \int_0^\infty \l 1-e^{-\lambda y} \r   \bar{\nu}(dy)  \r} \rr
\end{align}
and by performing the limit for $\gamma \to 0$ we arrive at
\begin{align}
\lim_{\gamma \to 0} \mathcal{L} \left[ \Pr \ll Z(\puntomio) > s \rr \right] (\lambda) \, = \, \frac{1}{\lambda} e^{-sf(\lambda)}.
\label{324}
\end{align}
Since \eqref{324} coincides with \eqref{cheroba} the proof is complete.
\end{proof}
\begin{os}
For $f(x) = x^\alpha$, $\alpha \in (0,1)$ result \eqref{convsubord} becomes
\begin{equation}
\lim_{\gamma \to 0} \sum_{j=0}^{N \l t \frac{t^{-\alpha}}{\Gamma (1-\alpha)} \r} Y_j \, \stackrel{\textrm{ law }}{=} \, \sigma^\alpha (t),
\label{convergsubstabile}
\end{equation}
where $\sigma^\alpha (t)$, $t\geq 0$, is the stable subordinator of order $\alpha \in (0,1)$ and the i.i.d. random variables $Y_j$ have power-law distribution
\begin{align}
\Pr \ll Y \in dy \rr /dy \, = \, \alpha \gamma^\alpha \,  y^{-\alpha -1} 	\, \mathds{1}_{y>\gamma}, \qquad \gamma > 0,
\end{align}
which can be obtained from \eqref{leggedelley} by performing the substitutions
\begin{align}
\bar{\nu} (y) \, = \, \frac{ \alpha y^{-\alpha -1}}{\Gamma (1-\alpha)} dy, \qquad \textrm{and} \qquad \nu(\gamma) \, = \, \frac{\gamma^{-\alpha}}{\Gamma (1-\alpha)},
\end{align}
due to the fact that $f(x) = x^\alpha = \eqref{bernsteinxallaalfa}$ ($a=0$, $b=0$).
The result \eqref{convergsubstabile} is well-known (see, for example, \citet{meer12}) and represents the convergence in distribution of a CTRW with power-law distributed jumps to the stable subordinator.
\end{os}

\section{Densities and related governing equations}
In this section we present in a unifying framework the governing equations of the densities of subordinators and their inverses, by making use of the operators defined in Section \ref{sezionederivate}.
\begin{te}
\label{teoremadensitasubeinv}
Let $^f\sigma(t)$, $t\geq 0$, and $^fL(t)$, $t\geq 0$, be respectively a subordinator and its inverse. Let $\bar{\nu}$ be the L\'evy measure such that $\bar{\nu}(0, \infty)=\infty$ and let $\nu(s) = a+ \bar{\nu}(s, \infty)$. Assume $s \to \nu(s)$ is absolutely continuous on $(0, \infty)$.
\begin{enumerate}
\item The probability density $\mu_t(x)$ of the subordinator $^f\sigma$ is the solution to the problem
 \begin{equation}
 \begin{cases}
 \frac{\partial}{\partial t} \mu_t(x) \, = \, - \, ^f\mathcal{D}_x^{(bt, +\infty)} \mu_t(x), \qquad & x > bt, 0<t< \infty, b \geq 0, \\
 \mu_t(bt) \, = \, 0, \qquad & 0< t < \infty, \\
 \mu_0(x) \, = \, \delta(x),
 \label{kernelsubord}
 \end{cases}
\end{equation}
\item The probability density $l_t(x)$ of $^fL(t)$, $t\geq 0$, is the solution to the equation
\begin{equation}
^f\mathcal{D}_t^{(0, \infty)} l_t(x) \, = \, - \frac{\partial}{\partial x} l_t(x), \qquad t>0, \textrm{ and } \begin{cases} 0<x<\frac{t}{b} < \infty, \quad &\textrm{if }b>0,  \\ 0<x<\infty, &\textrm{if }b=0, \end{cases}
\label{kernelinverso}
\end{equation}
subject to
\begin{align}
\begin{cases}
l_t (t/b) = 0, \\
l_t(0)= \nu(t), \\
l_0(x) = \delta(x).
\end{cases}
\label{condizioniinverso}
\end{align}
\end{enumerate}
The operator $^f\mathcal{D}_x^{(bt, +\infty)}$ is the one of Definition \ref{definizionederivata}.
\end{te}
\begin{proof}
As already pointed out the conditions assumed on $\bar{\nu}$ and $\nu(s)$ ensure that $\mu_t(B)$ and $l_t(B)$ are absolutely continuous and therefore have densities we denote again by $\mu_t(x)$ and $l_t(x)$.
\begin{enumerate}
\item First we note that $\mu_t(x) =0 $ for $x \leq bt$, $b \geq 0$, indeed from Proposition \ref{proposizrandomwalk}
\begin{align}
 \Pr \ll \, ^f\sigma(t) > bt  \rr \, = \, & \lim_{\gamma \to 0} \Pr \ll   bt +  \sum_{j=0}^{N \l t \nu (\gamma) \r } Y_j  > bt \rr \notag \\
 = \, & \lim_{\gamma \to 0} \Pr \ll \sum_{j=0}^{N \l t \nu (\gamma) \r } Y_j > 0 \rr \, = \, 1.
\end{align}
The Laplace transform of $\mu_t(x)$ is $\mathcal{L} \left[ \mu_t(\puntomio) \right] (\phi) \, = \, e^{-t f (\phi)}$ and therefore $\mathcal{L} \left[ \widetilde{\mu}_{\puntomio}(\phi) \right] (\lambda) = 1/(f(\lambda) + \phi)$. In view of Lemma \ref{laplacederivata} the Laplace transform of \eqref{kernelsubord} with respect to $x$ is
\begin{equation}
\frac{\partial}{\partial t} \widetilde{\mu}_t (\phi) \, = \, -f(\phi) \widetilde{\mu}_t(\phi) + b e^{-bt\phi} \mu_t(bt)
\label{laplaceb0} 
\end{equation}
and therefore by performing the Laplace transform with respect to $t$ we obtain
\begin{align}
\widetilde{\widetilde{\mu}}_\lambda (\phi) \, = \, \frac{1}{f(\lambda) + \phi}
\end{align}
where we used the facts that $\widetilde{\mu}_0(\phi) = 1$ and $\mu_t(bt)=0$. This completes the proof of (1).
\item First we show that $l_t(x) = 0$ for $x \geq \frac{t}{b}$ when $b>0$. By considering Proposition \ref{proposizrandomwalk} we have
\begin{align}
\Pr \ll \, ^fL(t) < \frac{t}{b} \rr \, = \, &  \Pr \ll ^f\sigma \l \frac{t}{b} \r > t \rr \notag \\
= \, & \lim_{\gamma \to 0} \Pr \ll t + \sum_{j=0}^{N \l \frac{t}{b} \nu (\gamma) \r} Y_j > t \rr \, = \, 1.
\end{align}
The function $t \to l_t(x)$ is differentiable since from Proposition \ref{lemmasullinverso} we have that $l_t(x) = b\mu_x(t)+\int_0^t\mu_x(t-z)\nu(z)dz$ and $t \to \mu_x(t)$ is differentiable in view of Theorem 28.1 of \cite{satolevy}.

The double Laplace transform of $l_t(x)$ reads
\begin{align}
\mathcal{L} \left[ \mathcal{L} \left[ l_t(x) \right] (\phi) \right](\lambda) \, = \, \frac{f(\lambda)/\lambda}{\phi + f(\lambda)},
\end{align}
where we used Proposition \ref{lemmasullinverso}.
From this point we temporary assume that $b>0$. We consider the Laplace transform with respect to $x$ of \eqref{kernelinverso} and we obtain
\begin{align}
^f\mathcal{D}_t^{(0, \infty)} \widetilde{l}_t(\phi) \, = \, -\phi \widetilde{l}_t(\phi) + l_t(0) - e^{-\phi (t/b)} l_t(t/b).
\label{laplaceprimavolta}
\end{align}
Considering the Laplace transform with respect to $t$ of \eqref{laplaceprimavolta} and by taking into account \eqref{condizioniinverso} we get
\begin{align}
f(\lambda) \widetilde{\widetilde{l}}_\lambda (\phi) - b \widetilde{l}_0 (\phi) \, = \, -\phi \widetilde{\widetilde{l}}_\lambda (\phi) + \frac{f(\lambda)}{\lambda} - b
\end{align}
where we used the fact that
\begin{align}
\int_0^\infty e^{-\lambda t} \nu(t)dt \, = \, \frac{f(\lambda)}{\lambda} - b
\end{align}
and Lemma \ref{laplacederivata}. The conditions \eqref{condizioniinverso} imply $\widetilde{l}_0(\phi) = 1$ and thus
\begin{equation}
\widetilde{\widetilde{l}}_\lambda (\phi) \, = \, \frac{f(\lambda)/\lambda}{\phi + f(\lambda)}.
\end{equation}
The proof for $b=0$ can be carried out equivalently.

\end{enumerate}
\end{proof}

\subsection{Some remarks on the long-range correlation}
The operators $^f\mathfrak{D}_t$ and $^f\mathcal{D}^{(c, \infty)}_t$ are non-local and govern processes with different memory properties. The presence of long-range correlation can be detected in several ways (see for example \citet{samoro}). Here we will explore the rate by which the correlation of the inverses of subordinators decays (a similar approach can be found in \citet{leone} applied to a fractional Pearson diffusion). In \citet{momentinv} the  authors derive an explicit formula for the moments of the inverse processes of subordinators. Define
\begin{align}
& \mathbb{E} \left[ \, ^fL(t_1)^{m_1} \, \cdots \, ^fL(t_n)^{m_n} \right] \, = \, U \l t_1, \dots, t_n; m_1, \dots, m_n \r.
\label{menomale}
\end{align}
Formula \eqref{menomale} obeys the recursion formula
\begin{align}
&U \l t_1, \dots, t_n; m_1, \dots, m_n \r \notag \\
= \, & \int_0^{t_{\min}} \sum_{i=1}^n m_i U \l t_1 - \tau, \dots, t_n - \tau, m_1, \dots, m_{i-1}, m_i -1, m_{i+1},\dots, m_n \r U(d\tau)
 \label{381}
\end{align}
where $t_{\min} = \min (t_1, \cdots, t_n)$. If $n=1$ and $m_1 =1$ the function \eqref{menomale} reduces to
\begin{equation}
U(x) = \mathbb{E} \left[ \, ^fL(x) \right] = \mathbb{E}\left[ \int_0^\infty \mathds{1}_{\ll \, ^f\sigma (t) \leq x \rr} dt \right]
\end{equation}
and is known as the renewal function since it is the distribution function of the renewal measure $U(dx)$. The  renewal measure is the potential measure of a subordinator and it is given by
\begin{equation}
U \l B \r \, = \, \mathbb{E} \int_0^\infty \mathds{1}_{[ \, ^f\sigma (t) \in B]} \, dt \, = \, \int_0^\infty \mu_t(B) \, dt, \qquad \textrm{for } B \subseteq [0, \infty),
\end{equation}
the reader can consults \citet{vondra} for further information.
We recall the renewal function is subadditive that is
\begin{align}
U(x+y) \leq U(x) + U(y), \qquad \forall x, y, \geq 0
\label{subadditive}
\end{align}
and that
\begin{equation}
\int_0^\infty e^{-\lambda x} U(dx) \, = \, \frac{1}{f(\lambda)} \qquad \int_0^\infty e^{-\lambda x} U(x) \, dx \, = \, \frac{1}{\lambda f(\lambda)}.
\end{equation}
Furthermore it is well-known (see, for example, \cite{bertoins}, Proposition 1.4) that there exist positive constants $c$ and $c^\prime$ such that
\begin{align}
c \, U(x) \leq \frac{1}{f\l \frac{1}{x} \r} \leq c^\prime \, U(x).
\label{bound}
\end{align}
By applying \eqref{381} we write
\begin{equation}
\mathbb{E} \, ^fL(s) \, ^fL(t)  \, = \, \int_0^{s \wedge t} \l U(s-\tau) + U(t-\tau)  \r U(d\tau)
\end{equation}
which can be interpreted as a long-range dependency property. We can write for $w>0$, 
\begin{align}
 \mathbb{E} \l \, ^fL(t) \, ^fL(t+s) \r  \, = \,  & \int_0^{t \wedge (t+s)} \l U(t-\tau) + U(t+s-\tau)  \r U(d\tau)  \notag \\
\geq \, &  \int_0^{t} U(s+2t-2\tau) U(d\tau) \notag \\
\geq \, &  \int_0^t \frac{1}{c^\prime f \l \frac{1}{s+2t-2\tau} \r} U(d\tau)
\label{centoedue}
\end{align}
where we applied \eqref{subadditive} and \eqref{bound}.
We recall that $1/f$ is monotone and thus we can write
\begin{align}
\lim_{s \to \infty} \int_0^t \frac{1}{c^\prime f \l \frac{1}{s+2t-2\tau} \r} U(d\tau) \, = \, \int_0^t \lim_{s \to \infty} \frac{1}{c^\prime f \l \frac{1}{s+2t-2\tau} \r} U(d\tau) >0
\label{finalmentesperiamo}
\end{align}
since $\lim_{z\to 0}f(z) \geq 0$. Fix $w,t>0$ and use formula \eqref{finalmentesperiamo}, we have
\begin{align}
\int_w^\infty \mathbb{E} \, ^fL(t) \, ^fL(t+s) \, ds \, = \, +\infty, \qquad \forall w,t>0.
\end{align}

\section{On the governing equations of time-changed $C_0$-semigroups}
In this section we discuss the concept of time-changed $C_0$-semigroups on a Banach space $(\mathfrak{B}, \left\| \puntomio \right\|_{\mathfrak{B}} )$ (see more on semigroup theory in \citet{autorivari, jacob1}) which  we define as the Bochner integral
\begin{equation}
\mathcal{T}_tu \, = \, \int_0^\infty T_su \, l_t(ds)
\end{equation}
where $T_s$ is a $C_0$-semigroup and $l_t$ is the distribution of the inverse $^fL(t)$, $t\geq 0$ of $^f\sigma (t)$, $t\geq 0$. We recall that a $C_0$-semigroup of operators on $\mathfrak{B}$ is a family of linear operators $T_t$ (bounded and linear) which maps $\mathfrak{B}$ into itself and is strongly continuous that is
\begin{align}
\lim_{t \to 0} \left\| T_tu - u \right\|_{\mathfrak{B}} \, = \, 0, \qquad  \forall u \in {\mathfrak{B}}.
\end{align}
In other words a bounded linear operator $T_t$ acting on a function $u \in {\mathfrak{B}}$ is said to be a $C_0$-semigroup if, $\forall u \in \mathfrak{B}$,
\begin{enumerate}
\item[$\bullet$]  $T_0u = u$ (is the identity operator),
\item[$\bullet$] $T_tT_su = T_sT_tu = T_{t+s}u$, $\forall s, t \geq 0$,
\item[$\bullet$] $\lim_{t \to 0} \left\| T_tu -u \right\|_{\mathfrak{B}} = 0$.
\end{enumerate}
The infinitesimal generator of a $C_0$-semigroup is the operator
\begin{align}
Au \, :=  \, \lim_{t \to 0} \frac{T_tu-u}{t},
\end{align}
for which
\begin{align}
\textrm{Dom} \l A \r := \ll u \in \mathfrak{B}: \lim_{t \to 0} \frac{T_tu-u}{t} \textrm{ exists as strong limit} \rr.
\end{align}
The aim of this section is to write the initial value problem associated with $\mathcal{T}_t$ by making use of the convolution-type time-derivatives of Definition \ref{definizaltern}.

\begin{te}
\label{teoremagoverninvsubord}
Let $^fL(t)$, $t\geq 0$, be the inverse process of a subordinator with Laplace exponent $f$ and let $l_t$ be the distribution of $^fL$. Let $\bar{\nu}(0, \infty) = \infty$ and $s \to \nu(s)= a+ \bar{\nu}(s, \infty)$ be absolutely continuous on $(0, \infty)$. Let $T_tu$, $u \in \mathfrak{B}$, be a (strongly continuous) $C_0$-semigroup on the Banach space $\l \mathfrak{B}, \left\| \puntomio \right\|_{\mathfrak{B}} \r$ such that $\left\| T_tu \right\|_{\mathfrak{B}} \leq \left\| u \right\|_\mathfrak{B}$. Let $\l A, \textrm{Dom}\l A \r \r$ be the generator of $T_tu$. The operator defined by the Bochner integral
\begin{equation}
\mathcal{T}_t u \, = \, \int_0^\infty T_su \, l_t(ds)
\label{312}
\end{equation}
acting on a function $u \in \mathfrak{B}$ is such that
\begin{enumerate}
\item  $\mathcal{T}_tu$ is a uniformly bounded linear operator on $\mathfrak{B}$,
\item $\mathcal{T}_tu$ is strongly continuous $\forall u \in  \mathfrak{B} $,
\item $\mathcal{T}_tu$ solves the problem
\begin{align}
\begin{cases}
^f\mathfrak{D}_t q(t) \, = \, A q(t), \qquad 0<t<\infty, \\
q(0) \, = \, u \in \textrm{Dom}\l A \r 
\end{cases}
\label{problemaimportante}
\end{align}
where the time-operator $^f\mathfrak{D}_t$ is the one appearing in Definition \ref{definizaltern}.
\end{enumerate}
\end{te}
\begin{proof}
Now we prove the Theorem for $b>0$ which is the case requiring some additional attention. The proof for $b=0$ can be carried out equivalently and therefore is a particular case. 
\begin{enumerate}
\item At first we show that the operator $\mathcal{T}_t u$ is uniformly bounded on $\l \mathfrak{B}, \left\| \puntomio \right\|_{\mathfrak{B}} \r$. From the hypothesys we have
\begin{equation}
\left\| T_t \right\| \leq 1, \qquad t \geq 0,
\label{esistecostante}
\end{equation}
In view of \eqref{esistecostante} we can write
\begin{align}
\left\| \mathcal{T}_t u \right\|_{\mathfrak{B}} \,  = \, & \left\| \int_0^\infty T_s u \, l_t (ds) \right\|_{\mathfrak{B}} \notag \\
\leq \, &  \int_0^\infty \left\| T_su  \right\|_{\mathfrak{B}} \, l_t(ds) \, \leq \,  \left\| u \right\|_{\mathfrak{B}},
\label{revision}
\end{align}
since $l_t[0, \infty) = 1$, $\forall t \geq 0$, as showed in Proposition \ref{lemmasullinverso}.
\item The strong continuity follows from the fact that
\begin{align}
\lim_{h \to 0} \left\| \mathcal{T}_{h}u - u  \right\|_{\mathfrak{B}} \, = \, & \left\| \int_0^\infty T_su \, l_h(ds) - u \right\|_{\mathfrak{B}} \notag \\
 \leq \, & \int_0^\infty \left\| T_su - u \right\|_{\mathfrak{B}} l_h(ds) \stackrel{h \to 0}{\longrightarrow} 0,
\end{align}
since $l_h \to \delta_0$ as $h \to 0$ and $T_s$ is strongly continuous.
\item Since $T_t$ is a $C_0$-semigroup generated by $\l A, \textrm{Dom}  A  \r$ we have
\begin{equation}
\frac{d}{d t} T_tu \, = \, A T_tu \, = \, T_t Au, \qquad \forall u \in \textrm{Dom} \l A \r.
\end{equation}
Now let
\begin{align}
A_s \, = \, \frac{T_su - u}{s}.
\end{align}
We note that
\begin{align}
A_s \mathcal{T}_tu \, = \, & A_s \int_0^\infty T_zu \, l_t(dz)  \notag \\
= \, & \int_0^\infty \frac{T_{z+s}u-T_zu}{s} \, l_t(dz) \notag \\
= \, & \int_0^\infty T_z \l \frac{T_su - u}{s} \r \, l_t(dz)
\end{align}
and since for $u \in \textrm{Dom}\l  A \r$ the limit for $s \to 0$ on the right-hand side exists we have that $\mathcal{T}_t$ maps $\textrm{Dom}\l A \r$ into itself.

By using Lemma \ref{lemmaplacealtern} we note that the Laplace transform of \eqref{problemaimportante} becomes
\begin{equation}
\begin{cases}
f(\lambda) \widetilde{q}(\lambda) - \frac{f(\lambda)}{\lambda} q( 0) \, = \, A \widetilde{q}( \lambda) \\
q(0) \, = \, u.
\end{cases}
\label{laplace(2)}
\end{equation}
Now define the operator
\begin{align}
^fR_{\lambda, A} \, : = \, \int_0^\infty e^{-\lambda t} \mathcal{T}_t dt \, = \, \frac{f(\lambda)}{\lambda} R_{f(\lambda), A}
\end{align}
where
\begin{equation}
R_{f(\lambda), A} = \int_0^\infty e^{-tf(\lambda)} T_t dt.
\label{aggia}
\end{equation}
We recall that since we assume $\l A, \textrm{Dom}\l A \r \r$ generate a $C_0$-semigroup for which $\left\| T_tu \right\|_\mathfrak{B} \leq \left\| u \right\|_\mathfrak{B}$, we necessarily have that $A$ is closed and densely defined. Furthermore for all $\lambda \in \mathbb{C}$ with $\Re \lambda >0$ we must have that $\lambda \in \rho (A)$ and $\left\| R_{\lambda, A} \right\| \leq \frac{1}{\Re \lambda}$, where
\begin{equation}
R_{\lambda, A} \, = \, \int_0^\infty e^{-\lambda t} T_t \, dt
\end{equation}
is the resolvent operator and $\rho(A)$ is the resolvent set of $A$. The integral \eqref{aggia} is justified since every Bernstein function has an extension onto the right complex half-plane $\mathbb{H}= \ll \lambda \in \mathbb{C} : \Re \lambda > 0 \rr$ which satisfies (see \cite{librobern}, Proposition 3.5)
\begin{align}
\Re f(\lambda) \, = \, a + b \Re \lambda  + \int_0^\infty \l 1-e^{-s \Re \lambda	} \cos \Im \lambda \r \bar{\nu}(ds) \, > \, 0.
\end{align}
By computing we can evaluate the following Laplace transform
\begin{align}
& \int_0^\infty e^{-\lambda t} \, ^f\mathfrak{D}_t \mathcal{T}_t u \, dt  \, = \notag \\ 
= \, &  \left[ b \int_0^\infty e^{-\lambda t}  \lim_{h \to 0} \frac{\mathcal{T}_{t+h}-\mathcal{T}_t}{h}  u \, dt +  \int_0^\infty e^{-\lambda t} \int_0^t \lim_{h \to 0} \frac{\mathcal{T}_{t+h-s}u-\mathcal{T}_{t-s}u}{h}   \, \nu(s) ds \, dt \right] \notag \\
= \, &  \left[ \lim_{h \to 0} b\frac{e^{\lambda h}}{h} \int_h^\infty e^{-\lambda t} \mathcal{T}_tu \, dt -b \lim_{h \to 0} \frac{1}{h} \int_0^\infty e^{-\lambda t} \mathcal{T}_tu \, dt \right. \notag \\
& \left. +  \int_0^\infty ds \, \nu(s) \int_s^\infty e^{-\lambda t} \lim_{h \to 0} \frac{\mathcal{T}_{t+h-s}u - \mathcal{T}_{t-s}u}{h} \right] \notag \\
= \, & \left[ b\lim_{h \to 0}\frac{e^{\lambda h }-1}{h} \, ^fR_\lambda u \, - b \lim_{h \to 0} \frac{e^{\lambda h}}{h} \int_0^h e^{-\lambda t} \mathcal{T}_tu \, dt \, \right. \notag \\
& \left. +  \l \frac{f(\lambda)}{\lambda} -b \r \l \lim_{h \to 0} \frac{e^{\lambda h}-1}{h}  \, ^fR_\lambda u - \lim_{h \to 0} \frac{e^{\lambda h}}{h} \int_0^h e^{-\lambda t} \mathcal{T}_tu \, dt \r \right] \notag \\
= \, &  \left[  \frac{f(\lambda)}{\lambda} \l \lim_{h \to 0} \frac{e^{\lambda h}-1}{h}  \, ^fR_\lambda u - \lim_{h \to 0} \frac{e^{\lambda h}}{h} \int_0^h e^{-\lambda t} \mathcal{T}_tu \, dt  \r \right] \notag \\
= \, & f(\lambda) \, ^fR_\lambda u - \frac{f(\lambda)}{\lambda}u,
\end{align}
where in the third step we used \eqref{gimisura}.

With this in hand we note that $^fR_{\lambda, A}$ satisfies
\begin{align}
\left\| \, ^fR_\lambda \right\| \, \leq \, \int_0^\infty \left\|  e^{-\lambda t} \mathcal{T}_t \right\| dt\, = \, \frac{1}{\Re \lambda},
\end{align}
where we used \eqref{revision}. Note that we can formally write
\begin{align}
\int_0^\infty e^{-\lambda t} \mathcal{T}_t dt \, = \, \frac{f(\lambda)}{\lambda} \int_0^\infty e^{-sf(\lambda)} T_s \, ds \, = \, & \frac{f(\lambda)}{\lambda} \int_0^\infty e^{-s(f(\lambda)-A)} ds \notag \\
 = \, & \frac{f(\lambda)}{\lambda} \frac{1}{f(\lambda) - A},
 \label{giustificare}
\end{align}
where we used Proposition \ref{lemmasullinverso} to state that $\mathcal{L} \left[ l_{\puntomio} (s) \right] (\lambda) \, = \, \frac{f(\lambda)}{\lambda} e^{-sf(\lambda)}$ and $l_t(s)$ represents by abuse of notation the density of $l_t(ds)$. In \eqref{giustificare} we used the exponential representation $T_t = e^{tA}$. Since we do not assume that $A$ is bounded the symbol $e^{tA}$ should be intended as $e^{tA}u = \textrm{strong-}\lim_{\lambda \to \infty} e^{tA_\lambda}u$ (Yosida approximation) where $A_\lambda := \lambda A R_\lambda$.

Now we have to prove that $\forall u \in \textrm{Dom}(A)$ we must have $^fR_\lambda u \in \textrm{Dom}(A)$ and
\begin{align}
\l f(\lambda) - A \r \, ^fR_\lambda u \, = \, \, ^fR_\lambda \l f(\lambda) - A \r u \, = \, \frac{f(\lambda)}{\lambda} u.
\end{align}
Now by the definition
\begin{align}
A_h \, = \, \frac{1}{h} \l T_hu - u \r
\end{align}
for which $\lim_{h \to 0} A_h = A$, we find
\begin{align}
A_h \, ^fR_\lambda u \, = \, & \frac{T_h - I}{h} \int_0^\infty e^{-\lambda t} \int_0^\infty T_su \, l_t(ds) \, dt \notag \\
 = \, & \int_0^\infty e^{-\lambda t} \int_0^\infty \frac{T_{s+h}u-T_su}{h} l_t(ds) \, dt \notag \\
= \, & \frac{f(\lambda)}{\lambda} \int_0^\infty e^{-sf(\lambda)} \frac{T_{s+h}u-T_su}{h} \, ds \notag \\
= \, & \frac{e^{hf(\lambda)}}{h} \frac{f(\lambda)}{\lambda} \int_h^\infty e^{-sf(\lambda)} T_zu \, dz \, - \frac{1}{h} \frac{f(\lambda)}{\lambda} \int_0^\infty e^{-sf(\lambda)} T_su \, ds \notag \\
= \, & \frac{f(\lambda)}{\lambda} \frac{e^{hf(\lambda)}-1}{h\lambda} \int_0^\infty e^{-zf(\lambda)} T_zu \, dz \, - \frac{f(\lambda)}{\lambda} \frac{1}{h} \int_0^h e^{-sf(\lambda)} T_su \, ds \notag \\
\stackrel{ h \to 0}{\longrightarrow} \, & f(\lambda) \, ^fR_\lambda u \, - \, \frac{f(\lambda)}{\lambda} u.
\end{align}
This proves that $^fR_\lambda u \in \textrm{Dom} \l A \r$ and that $\l f(\lambda) - A \r \, ^fR_\lambda u = \frac{f(\lambda)}{\lambda} u$. Furthemore we find
\begin{align}
^fR_\lambda A u  \, = \, &\int_0^\infty e^{-\lambda t} \mathcal{T}_tAu \, dt \, = \, \int_0^\infty e^{-\lambda t} \int_0^\infty T_sAu \, l_t(ds) \, dt \notag \\
= \, & \int_0^\infty e^{-\lambda t} \int_0^\infty \frac{d}{ds} T_su \, l_t(ds) \notag \\
= \, & \frac{f(\lambda)}{\lambda} \int_0^\infty e^{-sf(\lambda)} \frac{d}{ds} T_su \,  ds \notag \\
= \, & -\frac{f(\lambda)}{\lambda}u + f(\lambda) \, ^fR_\lambda u,
\end{align}
which completes the proof.
\end{enumerate}
\end{proof}

\subsection{Convolution-type space-derivatives and Phillips' formula}
\label{sezionefilippo}
Let $T_t$ be a $C_0$-semigroup acting on functions $u \in \mathfrak{B}$, where $\l \mathfrak{B}, \left\| \puntomio \right\|_\mathfrak{B} \r$ is Banach space. Let $\mu_t$ be a convolution semigroup of sub-probability measures on $[0, \infty)$ such that $\mathcal{L} [\mu_t] = e^{-tf}$ where $f$ is a Bernstein function. The operator defined by the Bochner integral
\begin{equation}
^fT_t u \, = \, \int_0^\infty T_su \, \mu_t(ds), \qquad u \in \mathfrak{B},
\end{equation}
is called a subordinate semigroup in the sense of Bochner.
A classical result due to \citet{filippo} state that the infinitesimal generator $\l \, ^fA, \textrm{Dom} \l \,^fA \r \r$ of the subordinate semigroup $^fT_t$ on $u \in \mathfrak{B}$ is written as
\begin{align}
 ^fA u \, = \, -f \l -A \r u\, = \,  \, -au +b A u + \int_0^\infty \l T_su -u\r \bar{\nu}(ds),
\label{piuomeno}
\end{align}
with $\textrm{Dom} \l A \r \subseteq \textrm{Dom} \l \, ^fA \r$.

In Definition \ref{definderivataspace} we developed the convolution-type space-derivatives $^f\mathpzc{D}_x^{\pm}$ defined on the whole real axis. We have shown that they becomes, for $f(x)= x^\alpha$, $\alpha \in (0,1)$, the Weyl space-fractional derivatives defined in \eqref{weyldestra} and \eqref{weylsinistra}. In this section we show that $- \, ^f\mathpzc{D}_x^-$ can be viewed as the infinitesimal generator of the subordinate semigroup in the sense of Bochner
\begin{align}
Q_tu(x) \, = \, \int_0^\infty T_s^lu(x) \mu_t(ds)
\end{align}
where $T_t^lu(x)=u(x+t)$, $u \in L^p \l \mathbb{R} \r$, is the left translation semigroup. 
\begin{os}
We recall that the left translation operator $T_t^lu = u(x+t)$, $t\geq 0$, $u \in L^p \l \mathbb{R} \r$, defines a strongly continuous $C_0$-semigroup on $L^p \l \mathbb{R} \r$ (see for example \cite{autorivari} page 66) and has infinitesimal generator $A = \frac{\partial}{\partial x}$ with $\textrm{Dom} \l A \r = W^{1, p}$, $1 \leq p < \infty$, where
\begin{equation}
W^{1, p} \l \mathbb{R} \r \, = \, \ll u \in L^p \l \mathbb{R} \r: u \textrm{ absolutely continuous and } u^\prime \in L^p \l \mathbb{R}  \r \rr.
\label{w1p}
\end{equation}
This implies that $- \, ^f\mathpzc{D}_x^-$ have to coincide with Phillips' representation \eqref{piuomeno} with $A = \frac{\partial}{\partial x}$.
\end{os}
\begin{prop}
Let $^f\sigma (t)$ be a subordinator with Laplace exponent $f$ and transition probabilities $\mu_t$. Let $\zeta = \inf \ll t\geq 0 : \, ^f\sigma (t) = + \infty \rr$.
The solution to the initial value problem
\begin{align}
\begin{cases}
\frac{\partial}{\partial t} q(x, t) \, = \, - \, ^f\mathpzc{D}_x^{-} \, q(x, t), \qquad x \in \mathbb{R}, 0<t<\infty, \\
q(x, 0) \,  = \,u(x) \in W^{1, p} \l \mathbb{R} \r, \\
\end{cases}
\label{(1)}
\end{align}
is given by the contractive strongly continuous semigroup of operators on $L^p \l \mathbb{R} \r$
\begin{equation}
Q_t u(x) \, = \,  \int_{0}^\infty u(x+y) \, \mu_t(dy), \qquad t < \infty
\label{semigrupposubordteorema}
\end{equation}
which is the subordinate translation semigroup $T_t^lu(x) = u(x+t)$, in the sense of Bochner.
The operator $^f\mathpzc{D}_x^-$ is that of Definition \ref{definderivataspace} and $W^{1,p}$ is defined in \eqref{w1p}.
\end{prop}
\begin{proof}
Since $Q_tu$ is a subordinate semigroup in the sense of Bochner, it defines again a $C_0$-semigroup on $L^p \l \mathbb{R} \r$.
By applying Phillips' result (\cite{filippo}) we know that the infinitesimal generator of $Q_tu$ is written as
\begin{align}
-f \l -\frac{\partial}{\partial x} \r u(x) \, = \, -au(x) + b \frac{\partial}{\partial x} u(x) + \int_0^\infty \l T_s^lu(x) - u(x) \r \bar{\nu}(ds).
\label{evabbe}
\end{align}
Since
\begin{align}
\left\| -f\l -\frac{\partial}{\partial x} \r u(x) \right\|_p \leq a \left\| u(x) \right\|_p + b \left\| \frac{\partial}{\partial x} u(x) \right\|_p + \int_0^\infty \left\| T_s^lu(x)-u(x) \right\|_p \bar{\nu}(ds)
\end{align}
by applying the well-known inequality (see for example \citet{jacob1})
\begin{equation}
\left\| T_tu(x) - u(x) \right\| \leq \l t \left\| Au(x) \right\| \wedge 2 \left\| u(x) \right\| \r, \qquad u \in \textrm{Dom}\l A \r
\label{disuguaglianzafiga}
\end{equation}
which is valid in general for a strongly continuous semigroup $T_tu(x)$ on a Banach space $\l \mathfrak{B}, \left\| \puntomio \right\| \r$ and infinitesimal generator $\l A, \textrm{Dom}\l A \r \r$, we can write
\begin{align}
\left\| -f \l -\frac{\partial}{\partial x} \r \right\|_p \leq & \, a\left\| u(x) \right\|_p +b \left\| \frac{\partial}{\partial x} u(x) \right\|_p + \int_0^\epsilon z \bar{\nu}(dz) \left\| \frac{\partial}{\partial x} u(x) \right\|_p  \notag \\
& + 2\int_\epsilon^\infty \bar{\nu}(dz) \left\| u(x) \right\|_p.
\label{nonimporta}
\end{align}
This shows that \eqref{evabbe} is defined on
\begin{align}
\begin{cases} W^{1,p} \l \mathbb{R} \r, \quad & \textrm{ if } b>0, \\ W^{1, p} \l \mathbb{R} \r, & \textrm{ if } b = 0 \textrm{ and } \bar{\nu}(0, \infty) = \infty, \\ L^p \l \mathbb{R} \r, & \textrm{ if }  b=0 \textrm{ and } \bar{\nu}(0, \infty) < \infty.  \end{cases}
\end{align}
since for $\bar{\nu}(0, \infty) < \infty$ we can choose $\epsilon = 0$ in \eqref{nonimporta}.

The operator $^f\mathpzc{D}_x^-$ 
\begin{align}
- \, ^f\mathpzc{D}_x^- u(x)  \, = \, &  b \frac{\partial}{\partial x} u(x) + \int_0^\infty \frac{\partial}{\partial x} u(x+s) \nu(s)ds
\end{align}
for $u \in W^{1, p} \l \mathbb{R} \r$ can be rewritten as
\begin{align}
- \, ^f\mathpzc{D}_x^- u(x) \, =  \, &   b \frac{\partial}{\partial x} u(x) + \int_0^\infty \frac{\partial}{\partial s} u(x+s) \l a+  \bar{\nu}(s, \infty) \r \, ds  \notag \\
= \, &  -au(x) + b \frac{\partial}{\partial x} u(x) + \int_0^\infty  \int_0^z \frac{\partial}{\partial s} T_s^lu(x) \, ds \, \bar{\nu}(dz) \notag \\
= \, & -au(x) + b \frac{\partial}{\partial x} u(x) + \int_0^\infty  \l T_z^lu(x) - u(x) \r \, \bar{\nu}(dz)
\end{align}
which coincides with \eqref{evabbe}. This completes the proof.
\end{proof}

\section{Example: the tempered stable subordinator}
By setting the Bernstein function considered in previous sections to be $f(x) = x^\alpha$, $\alpha \in (0,1)$, we retrive the stable subordinator $\sigma^\alpha (t)$, $t\geq 0$, for which $\mathbb{E}e^{-\lambda \sigma^\alpha(t)} = e^{-t\lambda^\alpha}$, and its inverse process $L^\alpha(t)$, $t\geq 0$. Therefore by performing the substitution $f(x) = x^\alpha$ all throughout the paper we retrive the results  related to fractional calculus. In this section we take as example the Bernstein function
\begin{equation}
f(x) \, = \, \l x+\vartheta  \r^\alpha -\vartheta ^\alpha \, = \,  \frac{\alpha}{\Gamma (1-\alpha)} \int_0^\infty  \l 1-e^{-x s} \r e^{-\vartheta  s} s^{-1-\alpha} \, ds,
\label{bernrelsub}
\end{equation}
where $\vartheta >0$, $\alpha \in (0,1)$.
The Bernstein function \eqref{bernrelsub} is the Laplace exponent of the subordinator $^{\vartheta }\sigma^\alpha (t)$ such that
\begin{align}
\mathbb{E}e^{-\lambda \, ^{\vartheta }\sigma^\alpha (t)} \, = \, e^{-t \l \l \lambda +\vartheta  \r^\alpha -\vartheta ^\alpha \r}.
\end{align}
The process $^{\vartheta }\sigma^\alpha (t)$, $t\geq 0$, is known in literature as the relativistic stable subordinator since it appears in the study of the stability of the relativistic matter (\citet{lieb}) but it is also known as the \emph{tempered stable subordinator} (see for example \citet{meer12} page 207, \citet{rosi} or \citet{zolo}, Lemma 2.2.1).
From \eqref{bernrelsub} we know that the L\'evy measure has the explicit representation
\begin{equation}
\bar{\nu} (ds)  \, = \, \frac{\alpha e^{-\vartheta s}s^{-\alpha -1}}{\Gamma (1-\alpha)}ds,
\end{equation}
and has infinite mass ($f(x)$ is not bounded). Furthermore its tail becomes
\begin{equation}
\nu(s) \, = \, \l \frac{\alpha \vartheta ^\alpha \Gamma (-\alpha, \, s)}{\Gamma (1-\alpha)}  \r,
\end{equation}
where
\begin{equation}
\Gamma (-\alpha, s) \, = \, \int_s^\infty e^{-z} z^{-\alpha -1} \, dz
\end{equation}
is the incomplete Gamma function.
It is well-known that the governing equation of $^{\vartheta }\sigma^{\alpha} (t)$, $t\geq 0$, is written by using the so-called \emph{tempered fractional derivative} 
\begin{align}
\partial^{\vartheta , \alpha}_x u(x) \, = \, e^{-\vartheta  x} \frac{^R\partial^\alpha}{\partial x^\alpha} \left[ e^{\vartheta  x} \, u(x) \right] - \vartheta ^\alpha u(x), \qquad \alpha \in (0,1),
\end{align}
as
\begin{equation}
\frac{\partial}{\partial t} \, \mu_t^{\vartheta, \alpha }(x) \, = \,- \partial_x^{\vartheta , \alpha} \, \mu_t^{\vartheta, \alpha }(x), \qquad x>0, t>0,
\end{equation}
see \cite{meer12} page 209 and the references therein. According to Theorem \ref{kernelsubord} we must have
\begin{equation}
\frac{\partial}{\partial t} \, \mu_t^{\vartheta, \alpha }(x) \, = \, - \,^f\mathcal{D}_x^{(0, \infty)} \, \mu_t^{\vartheta, \alpha }(x), \qquad x>0, t>0,
\end{equation}
and indeed it is easy to show that if $f(\lambda) = \l \lambda + \vartheta  \r^\alpha -\vartheta ^\alpha$
\begin{align}
 ^f\mathcal{D}_x^{(0, \infty)} u(x) \, = \, \frac{d}{d x} \int_0^x u(x-s) \l \frac{\alpha \vartheta ^\alpha \Gamma (-\alpha , s)}{\Gamma (1-\alpha)} \r \, ds \, = \, \partial^{\vartheta , \alpha}_x u(x).
\end{align}
This can be done for example by observing that
\begin{align}
\mathcal{L} \left[ \frac{d}{d x} \int_0^x u(x-s) \l \frac{\alpha \vartheta^\alpha \Gamma (-\alpha , s)}{\Gamma (1-\alpha)} \r \, ds \right] (\lambda) \, = \,  \mathcal{L} \left[ \partial^{\vartheta , \alpha}_x u(x) \right] (\lambda).
\end{align}

The time operator $^f\mathfrak{D}_t$ governing the density of 
\begin{align}
^{\vartheta }L^\alpha(t) \, = \, \inf \ll s>0 : \, ^{\vartheta }\sigma^\alpha (s) > t \rr,
\end{align}
becomes in this case
\begin{align}
^f\mathcal{D}_t^{(0, \infty)} l_t^{\vartheta, \alpha }(x) \,   = \, \frac{\partial}{\partial t}  \int_0^t  \, l_{t-s}^{\vartheta, \alpha }(x) \, \l \frac{\alpha \vartheta ^\alpha \Gamma (-\alpha, \, s)}{\Gamma (1-\alpha)}  \r \, ds,
\end{align}
and therefore $l_t^{\vartheta, \alpha }(x)$, $t>0$, is the solution to
\begin{align}
\begin{cases}
\frac{\partial}{\partial t} \int_0^t  \; l_{t-s}^{\vartheta, \alpha }(x) \, \l \frac{\alpha \vartheta ^\alpha \Gamma (-\alpha, \, s)}{\Gamma (1-\alpha)}   \r \, ds \, = \, - \frac{\partial}{\partial x} \; l_{t}^{\vartheta, \alpha }(x), \qquad & t>0, x>0, \\
l_{t}^{\vartheta, \alpha }(0) \, = \,  \frac{\alpha \vartheta ^\alpha \Gamma (-\alpha, \, t)}{\Gamma (1-\alpha)} , & t>0, \\
l_{0}^{\vartheta, \alpha }(x) \, = \, \delta(x).
\end{cases}
\end{align}

Finally, in view of Proposition \ref{proposizrandomwalk}, we are able to write  the CTRW converging in distribution to $^\vartheta \sigma^\alpha(t)$, $t\geq 0$. We have
\begin{align}
\lim_{\gamma \to 0} \sum_{j=0}^{N \l t \l  \frac{\alpha \vartheta ^\alpha \Gamma (-\alpha, \gamma)}{\Gamma (1-\alpha)} \r \r} Y_j \stackrel{\textrm{law}}{\longrightarrow} \, ^{\vartheta }\sigma^{\alpha} (t)
\end{align}
where $Y_j$ are i.i.d. random variables with distribution
\begin{align}
\Pr \ll Y_j \in dy \rr /dy \, = \, \frac{e^{-\vartheta y}y^{-\alpha-1}}{\vartheta ^\alpha \Gamma \l -\alpha, \gamma \r} \, \mathds{1}_{[y>\gamma]}, \qquad \gamma > 0, \, \forall j = 1, \dots, n,
\end{align}
and $N(t)$, $t\geq 0$, is an independent homogeneous Poisson process with parameter $\theta =1$.

\section{Ackwnoledgements}
Thanks are due to the referee whose remarks and suggestions have certainly improved a previous draft of the paper.


\begin{thebibliography}{16}
\providecommand{\natexlab}[1]{#1}
\providecommand{\url}[1]{\texttt{#1}}
\expandafter\ifx\csname urlstyle\endcsname\relax
  \providecommand{\doi}[1]{doi: #1}\else
  \providecommand{\doi}{doi: \begingroup \urlstyle{rm}\Url}\fi
  
  
\bibitem[Achar et al.(2007)]{achar}
B.N.N. Achar, C.F. Lorenzo and T.T. Hartley.
\newblock The Caputo fractional derivative: initialization issues relative to fractional differential equations.
\newblock \emph{Advances in fractional calculus}, 27--42, Springer, Dordrecht, 2007.
  
 
\bibitem[Baeumer and Meerschaert(2001)]{baem}
B. Baeumer and M.M. Meerschaert.
\newblock Stochastic solutions for fractional Cauchy problems.
\newblock \emph{Fractional Calculus and Applies Analysis}, 4(4): 481 -- 500, 2001.


\bibitem[Benson et al.(2001)]{benson}
D. Benson, R. Schumer, M. Meerschaert and S. Wheatcraft.
\newblock{Fractional dispersion, L\'evy motions, and the MADE tracer tests.}
\newblock{\emph{Transport in porous media}}, 42: 211 -- 240, 2001.


\bibitem [Bernstein(1929)]{artbern}
S. Bernstein.
\newblock {Sur les fonctions absolument monotones (french)}.
\newblock \emph{Acta Math}, 52; 1 -- 66, 1929.

\bibitem [Bertoin(1996)]{bertoinb}
J. Bertoin.
\newblock {L\'evy processes}.
\newblock \emph{Cambridge University Press}, Cambridge, 1996.

\bibitem [Bertoin(1997)]{bertoins}
J. Bertoin.
\newblock {Subordinators: examples and appications}.
\newblock \emph{Lectures on probability theory and statistics (Saint-Flour, 1997)}, 1 -- 91. \emph{Lectures Notes in Math.}, 1717, Springer, Berlin, 1999.



\bibitem [Bochner(1949)]{bochner}
S. Bochner.
\newblock {Diffusion equation and stochastic processes}.
\newblock \emph{Proc. Nat. Acad. Sci. USA}, 35: 368 -- 370, 1949.

\bibitem [Bochner(1955)]{bochner2}
S. Bochner.
\newblock {Harmonic analysis and the theory of probability}.
\newblock \emph{University of California Press}, Berkeley, 1955.




\bibitem[D'Ovidio(2012)]{mirkospa}
M. D'Ovidio.
\newblock{From Sturm-Liouville problems to fractional and anomalous diffusions.}
\newblock{\emph{Stochastic Processes and their Applications,}} 122: 3513 -- 3544, 2012.


\bibitem [Engel and Nagel(2000)]{autorivari}
K.J. Engel and R. Nagel.
\newblock{One-Parameter Semigroups for Linear Evolution Equations}.
\newblock \emph{Springer, Graduate Texts in Mathematics}, 2000.



\bibitem [Feller(1970)]{fellerb}
W. Feller.
\newblock {An introduction to probability theory and its applications, vol. II}.
\newblock \emph{Wiley}, New York, 1966.


\bibitem[Gorenflo et al.(1999)]{mainardiwright}
R. Gorenflo, Y. Luchko and F. Mainardi.
\newblock{Analytical properties and applications of the Wright function.}
\newblock{\emph{Fractional Calculus and applied analysis,}} 2(4): 383 -- 414, 1999.



\bibitem [It\^{o}(1942)]{itodec}
K. It\^{o}.
\newblock {On stochastic processes. I. (Infinitely divisible laws of probability).}
\newblock \emph{Japan J. Math.}, 18: 261 -- 301, 1942.



\bibitem[Jacob(2001)]{jacob1}
N. Jacob.
\newblock{Pseudo Differential Operators \& Markov Processes: Fourier analysis and semigroups. Volume 1.}
\newblock{\emph{Imperial College Press, Fourier analysis and semigroups}}, 2001.


\bibitem[Kilbas et al.(2006)]{kill}
A.A. Kilbas, H.M. Srivastava and J.J. Trujillo.
\newblock{Theory and Applications of Fractional Differential Equations.}
\newblock{\emph{North-Holland Mathematics Studies, 204. Elsevier Science B.V.,}} 2006.

\bibitem[Kochubei(2011)]{kochu}
A.N. Kochubei.
\newblock{General fractional calculus, evolution equations and renewal processes}.
\newblock{\emph{Integral Equations and Operator Theory}}, 71: 583 -- 600, 2011.


\bibitem[Leonenko et al.(2013))]{leone}
N. Leonenko, M.M. Meerschaert and A. Sikorskii.
\newblock{Correlation structure of fractional Pearson diffusions.}
\newblock{\emph{Computers \& Mathematics with Applications}}, In press, 2013.


\bibitem[Lieb(1990)]{lieb}
E.H. Lieb.
\newblock{The stability of matter: from atoms to stars.}
\newblock{\emph{Bull. Amer. Math.
Soc.}}, 22: 1 -- 49, 1990.


\bibitem[Lorenzo and Hartley(2002)]{lorenzo}
C.F. Lorenzo and T.T. Hartley.
\newblock{Variable order and distributed order fractional operators.}
\newblock{\emph{Nonlinear Dynam.}}, 29: 57 -- 98, 2002.



\bibitem [Mainardi(2010)] {mainalibro}
F. Mainardi.
\newblock {Fractional calculus and waves in linear viscoelasticity. An introduction to mathematical models.}
\newblock \emph{Imperial College Press}, xx+347 pp., London, 2010.




\bibitem [Meerschaert et al.(2009)] {meer09}
M.M. Meerschaert, E. Nane and P. Vellaisamy.
\newblock {Fractional Cauchy problems on bounded domains}.
\newblock \emph{The Annals of Probability}, 37(3):979 -- 1007, 2009.





\bibitem [Meerschaert and Scheffer(2004)] {meer04}
M.M. Meerschaert and H.P. Scheffer.
\newblock {Limit theorems for continuous time random
walks with infinite mean waiting times.}
\newblock \emph{J. Appl. Probab.}, 41: 623 -- 638, 2004.

\bibitem [Meerschaert and Scheffer(2008)]{meertri}
M.M. Meerschaert and H.P. Scheffer.
\newblock {Triangular array limits for continuous time random walks}.
\newblock \emph{Stochastic Processes and their Applications}, 118(9): 1606 -- 1633, 2008.

\bibitem [Meerschaert and Sikorskii(2012)] {meer12}
M.M. Meerschaert and A. Sikorskii.
\newblock {Stochastic Models for Fractional Calculus}.
\newblock \emph{Walter de Gruyter}, Vol. 43 of De Gruyter Studies in Mathematics, Berlin/Boston, 2012.


\bibitem [Meerschaert and Straka(2013)]{stracca}
M.M. Meerschaert and P. Straka.
\newblock {Inverse stable subordinators.}
\newblock \emph{Mathematical Modeling of Natural Phenomena}, 8(2): 1 -- 16, 2013.


\bibitem [Montroll and Weiss(1965)] {montrol}
E. Montroll and G.H. Weiss.
\newblock {Random walk on lattices. II.}
\newblock \emph{J. Math. Phys.}, 6: 167 -- 181, 1965.

\bibitem [Orsingher and Beghin(2003)] {orsptrf}
E. Orsingher and L. Beghin.
\newblock {Time-fractional telegraph equations and telegraph process with Brownian time}.
\newblock \emph{Probability Theory and Related Fields}, 128: 141 -- 160, 2003.



\bibitem [Orsingher and Beghin(2009)] {orsann}
E. Orsingher and L. Beghin.
\newblock {Fractional diffusion equations and processes with randomly varying time}.
\newblock \emph{The Annals of Probability}, 37(1):206 -- 249, 2009.








\bibitem [Phillips(1952)] {filippo}
R.S. Phillips.
\newblock {On the generation of semigroups of linear operators. Pacific J. Math.}.
\newblock \emph{Pacific J. Math.}, 2:343 -- 369, 1952.

\bibitem [Rosi\'nski(2007)] {rosi}
J. Rosi\'nski.
\newblock {Tempering stable processes}.
\newblock \emph{Stochastics Processes and their Applications}, 117: 677 -- 707, 2007.




\bibitem [Saichev and Zaslavsky(1997)]{saiche}
A.I. Saichev and G.M. Zaslavsky.
\newblock {Fractional kinetic equations: solutions and applications}.
\newblock \emph{Chaos}, 7(4): 753 -- 764, 1997.


\bibitem [Samorodnitsky(2006)]{samoro}
G. Samorodnitsky.
\newblock {Long range dependence}.
\newblock \emph{Foundation and Trends in Stochastic Systems}, 1(2): 163 -- 257, 2006.

\bibitem [Samorodnitsky and Taqqu(1944)]{samotaq}
G. Samorodnitsky and M. Taqqu.
\newblock {Stable non-Gaussian Random Processes}.
\newblock \emph{Chapman and Hall}, New York, 1944.




\bibitem [Sato(1999)] {satolevy}
K. Sato.
\newblock {L\'evy processes and infinitely divisible distributions}.
\newblock \emph{Cambridge University Press}, 1999




\bibitem [Schilling et al.(2010)]{librobern}
R.L. Schilling, R. Song and Z. Vondra\v{c}ek.
\newblock {Bernstein functions: theory and applications}.
\newblock \emph{Walter de Gruyter GmbH \& Company KG}, Vol 37 of De Gruyter Studies in Mathematics Series, 2010.

\bibitem[Song and Vondra\v{c}ek(2009)]{vondra}
R. Song and Z. Vondra\v{c}ek.
\newblock {Potential theory of subordinate Brownian motion}.
\newblock In: \emph{Potential Analysis of Stable Processes and its Extensions, P. Graczyk, A. Stos, editors, Lecture Notes in Mathematics 1980}: 87--176, 2009.



\bibitem[Veillette and Taqqu(2010)]{momentinv}
M. Veillette and M.S. Taqqu.
\newblock Using differential equations to obtain joint moments of first-passage times of increasing L\'evy processes.
\newblock \emph{Stat. Probab. Letters}, 80: 697 -- 705, 2010.


\bibitem[Zaslavsky(1994)]{zasla}
G. Zaslavsky.
\newblock Fractional kinetic equation for Hamiltonian chaos. Chaotic advection, tracer dynamics and turbulent dispersion.
\newblock \emph{Phys. D}, 76: 110 -- 122, 1994.


\bibitem[Zolotarev(1986)]{zolo}
V. Zolotarev.
\newblock One-dimensional stable distributions.
\newblock \emph{Translations of Mathematical Monographs 65, American Mathematical Society}, Providence, RI, 1986.




\end{thebibliography}
\end{document}